\documentclass[10pt]{amsart}

\usepackage[%
  tmargin=1in,bmargin=1in,%
  lmargin=1.65in,rmargin=1.65in,%
]{geometry}

\usepackage{lmodern}
\let\oldinfty\infty
\let\oldemptyset\emptyset

\usepackage{mathabx}
\let\infty\oldinfty
\let\emptyset\oldemptyset
\DeclareMathSymbol{*}{\mathbin}{symbols}{"03}

\usepackage{graphicx}        
\usepackage{multicol}        
\usepackage[bottom]{footmisc}
\usepackage{amscd,amsmath,amssymb,amsfonts}
\usepackage[textsize=footnotesize,backgroundcolor=orange!50]{todonotes}
\usepackage[all]{xy}
\usepackage[hidelinks]{hyperref}

\theoremstyle{definition}

\newtheorem{theorem}{Theorem}[section]
\newtheorem{proposition}[theorem]{Proposition}
\newtheorem{lemma}[theorem]{Lemma}
\newtheorem{corollary}[theorem]{Corollary}

\newtheorem{definition}[theorem]{Definition}
\newtheorem{notation}[theorem]{Notation}
\newtheorem{construction}[theorem]{Construction}
\newtheorem{variant}[theorem]{Variant}

\newtheorem{remark}[theorem]{Remark}

\newtheorem{example}[theorem]{Example}
\newtheorem{examples}[theorem]{Examples}
\numberwithin{equation}{theorem}

\usepackage{enumitem}

\setlist{%
  parsep=0ex, listparindent=\parindent,%
  itemsep=0.75ex, topsep=0.75ex,%
  leftmargin=2.5em,%
}

\setlist[enumerate, 1]{%
  label=(\roman*),%
  ref={\roman*},%
  widest=b,
}
\setlist[itemize, 1]{%
  label=--,%
}

\newcommand{\an}{{\rm an}}

\newcommand{\CH}{{\rm CH}}

\newcommand{\inj}{\hookrightarrow}

\newcommand{\End}{{\rm End}}
\newcommand{\Hom}{{\rm Hom}}

\newcommand{\Spec}{{\rm Spec\,}}

\newcommand{\Char}{{\rm char}}
\newcommand{\Tr}{{\rm Tr}}
\newcommand{\Hdg}{\mathrm{Hdg}}
\newcommand{\Gal}{{\rm Gal}}

\newcommand{\0}{\emptyset}

\newcommand{\sC}{{\mathcal C}}

\newcommand{\sE}{{\mathcal E}}
\newcommand{\sF}{{\mathcal F}}

\newcommand{\sK}{{\mathcal K}}

\newcommand{\sO}{{\mathcal O}}

\newcommand{\sU}{{\mathcal U}}

\newcommand{\sW}{{\mathcal W}}
\newcommand{\sX}{{\mathcal X}}
\newcommand{\sY}{{\mathcal Y}}

\newcommand{\A}{{\mathbb A}}

\newcommand{\C}{{\mathbb C}}

\newcommand{\G}{{\mathbb G}}

\renewcommand{\P}{{\mathbb P}}
\newcommand{\Q}{{\mathbb Q}}
\newcommand{\R}{{\mathbb R}}

\newcommand{\Z}{{\mathbb Z}}

\newcommand{\D}{\mathrm{D}}
\newcommand{\K}{\mathrm{K}}
\newcommand{\rG}{\mathrm{G}}
\newcommand{\M}{\mathrm{M}}
\newcommand{\rO}{\mathrm{O}}
\newcommand{\rR}{\mathrm{R}}
\newcommand{\rS}{\mathrm{S}}

\newcommand{\num}{{\operatorname{num}}}
\newcommand{\cyc}{{\operatorname{cyc}}}
\renewcommand{\det}{\operatorname{det}}

\newcommand{\id}{{\operatorname{id}}}

\newcommand{\Sch}{\mathrm{Sch}}

\newcommand{\op}{{\text{\rm op}}}
\newcommand{\<}{\langle}
\renewcommand{\>}{\rangle}

\renewcommand{\dim}{\operatorname{dim}}
 
\newcommand{\xr}[1]{\xrightarrow{#1}}

\newcommand{\Spc}{{\mathrm{Spc}}}

\newcommand{\lotimes}{\otimes^{\mathrm{L}}}
\newcommand{\Sm}{{\mathrm{Sm}}}

\newcommand{\Ab}{{\mathrm{Ab}}}

\newcommand{\Gr}{{\operatorname{\rm Gr}}} 
\newcommand{\rnk}{{\operatorname{\text{rnk}}}}

\newcommand{\W}{{\operatorname{W}}} 
\newcommand{\GW}{{\operatorname{GW}}} 
\newcommand{\sGW}{{\mathcal{GW}}}
\renewcommand{\H}{{\operatorname{H}}}
\newcommand{\SH}{{\operatorname{SH}}}
\newcommand{\T}{{\operatorname{T}}} 
\newcommand{\Th}{{\operatorname{Th}}} 
\renewcommand{\th}{{\operatorname{th}}} 

\newcommand{\Kos}{{\operatorname{Kos}}} 
\newcommand{\Aut}{{\operatorname{Aut}}}

\newcommand{\MW}{\mathrm{MW}}

\newcommand{\GL}{\operatorname{GL}}
\newcommand{\SL}{\operatorname{SL}}

\newcommand{\perf}{{\operatorname{perf}}}

\newcommand{\can}{\text{can}}

\newcommand{\ev}{\mathrm{ev}}

\DeclareMathOperator*{\iso}{{\simeq}}
\newcommand{\pr}{{\text{pr}}}
\newcommand{\BM}{\mathrm{B.M.}}

\newcommand{\SmL}{{\mathrm{Sm}\text{-}\mathrm{L}}}
\newcommand{\BiGr}{\mathrm{BiGr}}
\DeclareMathOperator{\HM}{\Pi_*}
\newcommand{\KGL}{{\operatorname{KGL}}}
\DeclareMathOperator{\Coh}{Coh}
\newcommand{\KO}{{\operatorname{KO}}} 
\newcommand{\BO}{{\operatorname{BO}}} 

\newcommand{\sRHom}{{\mathcal{RH}om}}

\newcommand{\ind}[1]{}

\newcommand{\inp}[1]{}

\newcommand{\bul}{{\raisebox{0.1ex}{\scalebox{0.6}{\hspace{0.08em}$\bullet$}}}}

\renewcommand{\top}{\mathrm{top}}

\begin{document}
\setcounter{tocdepth}{1}

\title{Motivic Gau{\ss}-Bonnet formulas}

\author{Marc Levine}
\address{Fakult\"at Mathematik\\
Universit\"at Duisburg-Essen\\
Thea-Leymann-Str. 9\\
45127 Essen\\
Germany}
\email{marc.levine@uni-due.de}

\author{Arpon Raksit}
\address{Department of Mathematics\\
Stanford University\\
450 Jane Stanford Way\\
Building 380\\
Stanford, CA 94305-2125\\
USA}
\email{arpon.raksit@gmail.com}



\begin{abstract}
  The apparatus of motivic stable homotopy theory provides a notion of Euler characteristic for smooth projective varieties, valued in the Grothendieck-Witt ring of the base field. Previous work of the first author and recent work of D\'eglise-Jin-Khan establishes a motivic Gau\ss-Bonnet formula relating this Euler characteristic to pushforwards of Euler classes in motivic cohomology theories. In this paper, we apply this formula to $\SL$-oriented motivic cohomology theories to obtain explicit characterizations of this Euler characteristic. The main new input is a uniqueness result for pushforward maps in $\SL$-oriented theories, identifying these maps concretely in examples of interest.
\end{abstract}

\maketitle

\tableofcontents

\section{Introduction}
\label{intro}

Let $k$ be a field and let $X$ be a smooth projective $k$-scheme. Let $\SH(k)$ denote the motivic stable homotopy category over $k$; recall that this comes equipped with the structure of a symmetric monoidal category, whose tensor product we denote $\wedge_k$ and whose unit object (the motivic sphere spectrum) we denote $1_k$.

Our starting point in this paper is the following fact, which shall be reviewed in \S\ref{sec:Dual}, and which goes back to the categorical notion of Euler characteristic introduced by Dold-Puppe \cite{DoldPuppe}.

\begin{proposition}
  \label{prop:IntroDual}
  The infinite suspension spectrum $\Sigma^\infty_\T X_+ \in \SH(k)$ is dualizable. In particular, we can associate to $X$ a natural \emph{Euler characteristic} $\chi(X/k) \in \End_{\SH(k)}(1_k)$, defined as the composition
  \[
    1_k \xr{\delta} \Sigma^\infty_\T X_+ \wedge_k (\Sigma^\infty_\T X_+)^\vee \xr{\tau} (\Sigma^\infty_\T X_+)^\vee \wedge_k \Sigma^\infty_\T X_+ \xr{\epsilon} 1_k,
  \]
  where the maps $\delta$ and $\epsilon$ are the coevaluation and evaluation that comprise the duality, and $\tau$ is the symmetry isomorphism.
\end{proposition}

For $k$ perfect, a theorem of Morel identifies $\End_{\SH(k)}(1_k)$ with the Grothendieck-Witt group $\GW(k)$, i.e. the Grothendieck group of $k$-vector spaces equipped with a nondegenerate symmetric bilinear form. Hence, we may think of the Euler characteristic $\chi(X/k)$ as a class in $\GW(k)$. It is then natural to wonder whether there is an explicit interpretation of this Euler characteristic in terms of symmetric bilinear forms. An intuitive speculation is that the Euler characteristic should be given by the value of some cohomology theory on $X$, equipped with an intersection pairing.

One of the main results in this paper is to make precise and confirm this speculation. To state the result, we recall that classes in $\GW(k)$ can be represented not just by nondegenerate symmetric bilinear forms on $k$-vector spaces, but also by nondegenerate symmetric bilinear forms on perfect complexes over $k$ (see \S\ref{sec:KO} for a review of what this means). With this in mind, we give the following explicit interpretation of the Euler characteristic $\chi(X/k)$.

\begin{construction}
 Suppose that $X$ is of pure dimension $d$. Then we have the \emph{Hodge cohomology} groups $\H^i(X;\Omega^j_{X/k})$ for $0 \le i,j \le d$ and the canonical trace map
  \[
    \Tr : \H^d(X;\Omega^d_{X/k}) \to k.
  \]
  We define a perfect complex of $k$-vector spaces (with zero differential),
  \[
    \Hdg(X/k) := \bigoplus_{i,j=0}^d \H^i(X,\Omega^j_{X/k})[j-i],
  \]
  and the trace map defines a nondegenerate symmetric bilinear form on $\Hdg(X/k)$ via the pairings
  \[
    \H^i(X,\Omega^j_{X/k}) \otimes_k \H^{d-i}(X,\Omega^{d-j}_{X/k}) \xr{\cup} \H^d(X,\Omega^d_{X/k}) \xr{\Tr} k,
  \]
  where the first map denotes the cup product (that this is indeed a nondegenerate symmetric bilinear form will be shown in \S\ref{sec:KO}). We thus obtain a Grothendieck-Witt class $(\Hdg(X/k),\Tr) \in \GW(k)$. This construction extends in an evident manner to the case that $X$ is not necessarily of pure dimension.
\end{construction}

The following formula for $\chi(X/k)$ was proposed by J-P.~Serre.\footnote{private communication to ML, 28.07.2017}

\begin{theorem}
  \label{thm:IntroHodge}
  Assume that $k$ is a perfect field of characteristic different from two. 
  Then $\chi(X/k) = (\Hdg(X/k), \Tr) \in \GW(k)$.
\end{theorem}

We in fact prove a more general result over a base-scheme $B$; see 
Theorem~\ref{thm:EulerCharKO}  and Corollary~\ref{cor:EulChar1} for details.

In the case $k=\R$, a class in $\GW(k)$ is determined by two $\Z$-valued invariants, rank and signature, and Theorem~\ref{thm:IntroHodge} reproves the following known result (see \cite[Theorem 1]{Abelson} and \cite[Theorem A]{Kharlamov}).

\begin{corollary}
  \label{cor:IntroReal}
  Suppose that $k=\R$ and $X$ is of even pure dimension $2n$. Then the symmetric bilinear form
  \[
    \H^n(X,\Omega^n_{X/\R}) \times \H^n(X,\Omega^n_{X/\R}) \xr{\cup} \H^{2n}(X,\Omega^{2n}_{X/\R}) \xr{\Tr} \R
  \]
  has signature equal to $\chi^\top(X(\R))$, the classical Euler characteristic of the real points of $X$ in the analytic topology. In particular, we have
  \[
    |\chi^\top(X(\R))| \le \dim_\R \H^n(X,\Omega^n_{X/\R}).
  \]
\end{corollary}
This is Corollary~\ref{cor:EulChar2} in the text.

Besides giving an explicit formula for the rather abstractly defined $\chi(X/k)$, Theorem~\ref{thm:IntroHodge} opens the way to computing $\chi(X/k)$ in the situation that $X$ is a twisted form of another $k$-scheme $Y$, namely by twisting the symmetric bilinear form $(\Hdg(Y/k), \Tr)$ by the descent data for $X$. This cannot be done with the class $\chi(Y/k)\in \GW(k)$, as $\GW(-)$ does not satisfy Galois descent. This is discussed in detail in \S\ref{subsec:Descent}.

Let us now explain our methods for proving Theorem~\ref{thm:IntroHodge}. The idea is to use the theory of Euler classes in motivic cohomology theories. More specifically, our focus is on cohomology theories represented by $\SL$-oriented motivic ring spectra; recall that this refers to a commutative monoid object $\sE$ in $\SH(k)$ equipped with a compatible system of Thom classes for oriented vector bundles (where an orientation is a specified trivialization of the determinant line bundle). The example of interest for proving Theorem~\ref{thm:IntroHodge} is hermitian K-theory; other examples of interest include Chow-Witt theory, ordinary motivic cohomology, and algebraic K-theory (the last two are actually $\GL$-oriented, meaning they have Thom classes for all vector bundles).  The assumption that $k$ has characteristic different from two in Theorem~\ref{thm:IntroHodge} arises from  this use of hermitian K-theory, which at present is only known to satisfy the properties we need when $\Char(k) \neq 2$; we do not know of  any counter-examples to our formula for $\chi(X/k)$ for $k$ of characteristic two.

Given an $\SL$-oriented motivic ring spectrum $\sE \in \SH(k)$, one may define certain pushforward maps in twisted $\sE$-cohomology. Namely, if $Y$ and $Z$ are smooth quasi-projective $k$-schemes, $f : Z \to Y$ is a proper morphism of relative dimension $d \in \Z$, and $L$ is a line bundle on $Y$, then there is a pushforward map
\[
  f_* : \sE^{a,b}(Z;\omega_{Z/k} \otimes f^*L) \to \sE^{a-2d,b-d}(Y;\omega_{Y/k}\otimes L),
\]
where $\omega_{-/k}$ denotes the canonical bundle. This is defined abstractly via the six-functor formalism for motivic stable homotopy theory.

We note two key examples of these pushforwards, assuming our smooth projective variety $X$ is of pure dimension $d$ for simplicity:
\begin{itemize}
\item the structural morphism $\pi : X \to \Spec(k)$ gives a pushforward map
  \[
    \pi_* : \sE^{2d,d}(X,\omega_{X/k}) \to \sE^{0,0}(\Spec k);
  \]
\item given a vector bundle $p : V \to X$, the zero section $s : X \inj V$ gives a pushforward map
  \[
    s_* : \sE^{0,0}(X) \to \sE^{2d,d}(V;p^*\det^{-1}(V)).
  \]
\end{itemize}
The first should be thought of as a kind of integration map. The second allows us to define the \emph{Euler class} of a vector bundle $V \to X$,
\[
  e^\sE(V) := s^*s_*(1) \in \sE^{2d,d}(X;\det^{-1}(V)),
\]
where $s$ again denotes the zero-section, and $1 \in \sE^{0,0}(X)$ denotes the unit element.

Using the above notions, we may state the following motivic version of the classical Gau\ss-Bonnet formula equating the Euler characteristic with the integral of the Euler class of the tangent bundle; this result is a fairly immediate consequence of  \cite[Lemma 1.5]{LevEnum}:

\begin{theorem}[Motivic Gau\ss-Bonnet]
  \label{thm:IntroGaussBonnet}
  Let $\sE$ be an $\SL$-oriented motivic ring spectrum in $\SH(k)$. Let $u : 1_k \to \sE$ denote the unit map, inducing the map $u_* : \GW(k) \simeq 1_k^{0,0}(\Spec k) \to \sE^{0,0}(\Spec k)$. Then we have
  \[
    u_*(\chi(X/k)) = \pi_*(e^\sE(T_{X/k})) \in \sE^{0,0}(\Spec k).
  \]
\end{theorem}
A general motivic Gau\ss-Bonnet formula is also proven in \cite[Theorem 4.6.1]{DJK}, which implies the above formula by applying the unit map. Our method is somewhat different from \cite{DJK} in that we replace their general theory of Euler classes with the more special version for $\SL$-oriented theories used here; see Theorem~\ref{thm:GaussBonnet} below for our general statement of this result.

As stated above, we deduce Theorem~\ref{thm:IntroHodge} from Theorem~\ref{thm:IntroGaussBonnet} by considering the example of hermitian K-theory, $\sE = \BO$. In this case, the map $u_* : \GW(k) \to \BO^{0,0}(\Spec k)$ is an isomorphism. The deduction requires an explicit understanding of both the Euler class $e^\sE(T_{X/k})$ and the pushforward $\pi_*$ in hermitian K-theory; the former is fairly straightforward, but the latter requires new input.

What we do is identify the abstractly defined projective pushforward maps in hermitian K-theory with the concrete ones defined in terms of pushforward of sheaves and Grothendieck-Serre duality. This comparison follows from a uniqueness result we prove for pushforward maps in an $\SL$-oriented theory $\sE$, characterizing them, under certain further hypotheses on $\sE$, in terms of their behavior in the case of the inclusion of the zero-section of a vector bundle (which is governed by Thom isomorphisms). We leave the detailed statement of this result to the body of the paper (see Theorem~\ref{thm:SLUniqueness}), as it would take too long to spell out here.

\begin{remark}
  \label{rmk:BachmannWickelgren}
  The recent paper of Bachmann--Wickelgren \cite{BachmannWickelgren} discusses results closely related to those discussed here. For example, they identify the abstract pushforward maps in hermitian K-theory with those defined by Grothendieck-Serre duality in the case of a finite syntomic morphism (as opposed to the case of a smooth and proper morphism between smooth schemes addressed here). Moreover, combining their identifications of various Euler classes with our motivic Gauss-Bonnet formula, one may recover Theorem~\ref{thm:IntroHodge} above.
\end{remark}

\subsection*{Outline}

The paper is organized as follows. In \S\ref{sec:Dual}, we review the basic framework of motivic homotopy theory, as well as relevant aspects of the dualizability result Proposition~\ref{prop:IntroDual}. In \S\ref{sec:SL}, we review basic facts about $\SL$-oriented motivic ring spectra. In \S\ref{sec:Push}, we describe the abstractly defined pushforwards in the twisted cohomology theory arising from an $\SL$-oriented motivic ring spectrum. In \S\ref{sec:GaussBonnet}, we prove the general Gau\ss-Bonnet formula for $\SL$-oriented motivic ring spectra. In \S\ref{sec:SLOrCohThy}, we axiomatize the features of the twisted cohomology theory arising from an $\SL$-oriented motivic ring spectrum. In \S\ref{sec:Uniq}, we use these axiomatics to prove our unicity/comparison theorem characterizing the pushforward maps in $\SL$-oriented theories. And finally, in \S\ref{sec:App}, we apply the previous results in specific examples of $\SL$-oriented theories to obtain various concrete consequences, in particular proving Theorem~\ref{thm:IntroHodge}.

\subsection*{Acknowledgements}

Both authors thank the referees for their detailed comments and suggestions, which helped us correct a number of ambiguities in an earlier version and have greatly improved the exposition. 

AR: I would like to thank my PhD advisor, S\o ren Galatius, for many suggestions and ideas that led to my thinking about the questions discussed in this paper. I would also like to thank Tony Feng and Jesse Silliman for helpful conversations. My work has been supported by an NSF Graduate Research Fellowship.

ML: I would like to thank J-P.~Serre for proposing the formula Theorem~\ref{thm:IntroHodge} for the motivic Euler characteristic and J-P.~Serre and Eva Bayer-Fluckiger  for an enlightening correspondence on computations in the case of geometrically rational surfaces.  My work has been supported by the DFG through the SFB Transregio 45 and the SPP 1786.

\section{Duality and Euler characteristics}
\label{sec:Dual}

In this section, we review the strong dualizability of smooth projective schemes as objects of the stable motivic homotopy category, which supplies a notion of Euler characteristic for these schemes. We additionally recall a result from \cite{LevEnum} that gives an alternative characterization of this Euler characteristic.

\subsection{Preliminaries}

Let us first recall the basic framework of stable motivic homotopy theory, which will be used throughout the paper.

\begin{notation}
  Throughout the paper, we let $B$ denote a noetherian separated base scheme of finite Krull dimension. Let $\Sch_B$ denote the category of quasi-projective $B$-schemes, that is, $B$-schemes $X \to B$ that admit a closed immersion $i : X \inj U$ over $B$, with $U$ an open subscheme of $\P^N_B$ for some $N$. Let $\Sch_B^\pr$ denote the subcategory of $\Sch_B$ with the same objects as $\Sch_B$ but with morphisms the proper morphisms. Let $\Sm_B$ denote the full subcategory of $\Sch_B$ with objects the smooth (quasi-projective) $B$-schemes. (The same notation will be used when working over schemes other than $B$.) For $X$ a $B$-scheme, we will usually denote the structure morphism by $\pi_X:X\to B$.
\end{notation}

\begin{notation}
  \label{ntn:SH}
  Given $X \in \Sch_B$, we let $\SH(X)$ denote the stable motivic homotopy category over $X$. We will rely on the six-functor formalism for this construction, as established in \cite{Ayoub, Hoyois6}. In particular, for each morphism $f : Y \to X$ in $\Sch_B$, one has the adjoint pairs of functors
  \[
    \xymatrix{\SH(X)\ar@<3pt>[r]^{f^*}&\SH(Y)\ar@<3pt>[l]^{f_*}}
    \quad\text{and}\quad
    \xymatrix{\SH(Y)\ar@<3pt>[r]^{f_!}&\SH(X)\ar@<3pt>[l]^{f^!}};
  \]
  natural isomorphisms $(gf)^*\simeq f^*g^*$, $(gf)^!\simeq f^!g^!$, $(gf)_*\simeq g_*f_*$, $(gf)_!\simeq g_!f_!$ for composable morphisms, with the usual associativity; a natural transformation $\eta^f_{!*}:f_!\to f_*$, which is an isomorphism if $f$ is proper. There are various base-change morphisms, which we will recall as needed. In addition, for $f$ smooth, there is a further adjoint pair
  \[
    \xymatrix{\SH(Y)\ar@<3pt>[r]^{f_\sharp}&\SH(X)\ar@<3pt>[l]^{f^*}}.
  \]
There is also the symmetric monoidal structure on $\SH(X)$; we denote the tensor product by $\wedge_X$ and the unit by $1_X\in \SH(X)$. For $f$ a closed immersion, we have the adjoint pair $f_*\dashv f^!$ arising from a corresponding adjoint pair in the unstable setting, so we will take $f_!=f_*$ with $\eta^f_{!*}=\id$. Similarly,  if $f$ is an open immersion, we have a canonical isomorphism of adjoint pairs $(f_\#\dashv f^*)\cong (f_!\dashv f^!)$, so we take $f_!=f_\#$ and $f^!=f^*$.

  We also have the unstable motivic homotopy category $\H_\bul(X)$, which we recall is a localization of the category $\Spc_\bul(X)$ of presheaves of pointed simplicial sets on $\Sm_X$. For $Y\to X$ in $\Sm_X$, we write $Y_+$ for the presheaf represented by the $X$-scheme $Y\amalg X\to X$, that is, the presheaf $Z\mapsto \Hom_{\Sm_X}(X, Y)_+$ (where here $(-)_+$ denotes addition of a disjoint basepoint to a set and we regard a set as a constant simplicial set). The category $\Spc_\bul(X)$ has a canonical symmetric monoidal structure with unit object $X_+$, and $\H_\bul(X)$ inherits this structure via the localization functor. Finally, we have the infinite suspension functor $\Sigma^\infty_\T:\H_\bul(X)\to \SH(X)$, which is canonically symmetric monoidal, so that in particular we have a canonical identification $1_X \simeq \Sigma^\infty_\T X_+$.

  For $f : Y \to X$ a smooth morphism in $\Sch_B$, the adjoint pair $f_\sharp \dashv f^*$ mentioned above for stable motivic homotopy categories arises from an adjoint pair in the unstable setting,
  \[
    \xymatrix{\H_\bul(Y)\ar@<3pt>[r]^{f_\sharp}&\H_\bul(X)\ar@<3pt>[l]^{f^*}},
  \]
  where the left adjoint is induced by the functor $f_\#:\Spc_\bul(Y)\to \Spc_\bul(X)$ obtained as the left Kan extension of the functor $f_\#:\Sm_Y\to \Spc_\bul(X)$ sending $p:W\to Y$ to $(f\circ p:W\to X)_+$.  

  We often write $-/X$ for the functor $\Sigma^\infty_\T:\H_\bul(X)\to \SH(X)$, and if $\pi_X:X\to B$ is smooth, then we write $-/B$ for the functor $\pi_{X\#}\circ\Sigma^\infty_\T:\H_\bul(X)\to \SH(B)$. We use the same notation to denote the precompositions of these functors with the functor $(-)_+ : \Sm_X \to \H_\bul(X)$.
\end{notation}

\begin{remark}
  For $p : E \to X$ an affine space bundle in $\Sch_B$ (that is, a torsor in the Zariski topology for a vector bundle), the composition $p_\sharp \circ p^*$ is an autoequivalence of $\SH(X)$ (this is a formulation of homotopy invariance). 
\end{remark}

\begin{notation}[The localization triangle]
  Let $j:U\to X$ be an open immersion in $\Sch_B$ with (reduced) complement $i:Z\to X$. We have the {\em localization distinguished triangle} of endofunctors on $\SH(X)$
  \begin{equation}\label{eqn:Localization}
    j_!j^!\to \id_{\SH(X)}\to i_*i^*\to j_!j^![1],
  \end{equation}
  where the morphism $j_!j^!\to \id_{\SH(X)}$ is the counit of the adjunction $j_!\dashv j^!$ and the morphism $\id_{\SH(X)}\to i_*i^*$ is the unit of the adjunction $i^*\dashv i_*$.  Moreover $i_*=i_!$, $j_!=j_\#$ and $j^!=j^*$. 

  We often write $X_Z/X$ for $i_*(1_Z)\in \SH(X)$. With this notation (and that of Notation~\ref{ntn:SH}), applying \eqref{eqn:Localization} to $1_X$ gives us the distinguished triangle in $\SH(X)$,
  \[
    U/X\xr{j/X} X/X\to X_Z/X\to U/X[1];
  \]
  in other words, we have a canonical isomorphism $X_Z/X \simeq \Sigma^\infty_T(X/U)$; accordingly, we often write $X_Z$ for the quotient presheaf $X/U$ in $\H_\bul(X)$. 
\end{notation}

\begin{notation}[Suspension and Thom spaces]
  Let $p : V \to X$ be a vector bundle over some $ X \in \Sch_B$, with zero-section $s : X \inj V$. We have the endofunctors
  \[
    \Sigma^{-V}, \Sigma^V:\SH(X)\to \SH(X)
  \]
  defined by $\Sigma^{-V}:=s^!p^*$ and $\Sigma^V:=p_\sharp\circ s_*$.
  These are in fact inverse autoequivalences. 

  The endofunctor $\Sigma^V$ can also be written in terms of Thom spaces. Setting $0_V := s(X) \subset V$, the Thom space of the vector bundle is defined as
  \[
    \Th_X(V) := V/(V \setminus 0_V) \in \H_\bul(X).
  \]
  To lighten the notation, we often write $\Th_X(V)$ for $\Th_X(V)/X=\Sigma^\infty_\T(\Th_X(V)) \in \SH(X)$, when the context makes the meaning clear. For $\pi_X:X\to B$ in $\Sm_B$, we set $\Th(V):=\pi_{X\#}(\Th_X(V))$ in $\H_\bul(B)$, and we similarly write $\Th(V)$ for $\Th(V)/B \simeq \pi_{X\#}(\Th_X(V)/X)\in \SH(B)$ when appropriate.
  
  To see the relation between the Thom space $\Th_X(V)$ and the suspension functor $\Sigma^V$, consider the localization distinguished triangle
  \[
    j_!j^!\to \id_V\to s_*s^*\to j_!j^![1],
  \]
  where $s$ still denotes the zero section and $j$ denotes the open complement $V \setminus 0_V \inj V$. Evaluating the sequence at $1_V$ and noting that $j_!j^!=j_\#j^*$ and  $s_*=s_!$, we obtain an
  identification between $\Th_X(V)\in \SH(X)$ and $p_\#s_*(1_X)=\Sigma^V(1_X)$. Consequently, we have identifications $\pi_{V\#}(s_*(1_X))\simeq \pi_{X\#}(\Sigma^V(1_X))\simeq \Th(V)$.
  
  In parallel, we shall write $\Th_X(-V)$ for $\Sigma^{-V}(1_X)$ and $\Th(-V)$ for $\pi_{X\#}\Sigma^{-V}(1_X)$. With these notational conventions, there are canonical natural isomorphisms
  \begin{equation}
    \label{eq:SuspensionSmashing}
    \Sigma^V(-)\simeq \Th_X(V)\wedge_X(-), \quad
    \Sigma^{-V}(-)\simeq \Th_X(-V)\wedge_X(-).
  \end{equation}
\end{notation}

\begin{remark}
  Let $\D_\perf(X)_{\mathrm{iso}}$ denote the subcategory of isomorphisms in the perfect derived category $\D_\perf(X)$ and let $\sK(X)$ denote the groupoid associated to the K-theory space of $X$. Then the assignment $V \mapsto \Sigma^V$ extends to a functor
  \[
    \Sigma^{(-)} :\D_\perf(X)_{\mathrm{iso}} \to \Aut(\SH(X)),
  \]
  and moreover factors through the canonical functor $\D_\perf(X)_{\mathrm{iso}} \to \sK(X)$, so that a distinguished triangle $E'\to E\to E''\to E'[1]$ in $\D_\perf(X)$ determines a natural isomorphism $\Sigma^{E'}\circ\Sigma^{E''}\simeq \Sigma^E$.  See for example \cite[Proposition 4.1.1]{Riou} for a proof of this last statement in the special case concerning the functor $\Th_X(-)=\Sigma^{(-)}(1_X) : 
  \D_\perf(X)_{\mathrm{iso}} \to  \SH(X)$ (which in fact implies the general statement by \eqref{eq:SuspensionSmashing}).

  In this context, for an integer $n$, we sometimes write $n$ for the trivial bundle of virtual rank $n$; for example, $\Sigma^{n+E} \simeq \Sigma^n_{\P^1}\circ\Sigma^E$ for $E$ in $\D_\perf(X)$.
\end{remark}

\begin{remark}[Atiyah duality]
  \label{rmk:Atiyah}
  For $f:Y\to X$ a smooth morphism in $\Sch_B$ with relative tangent bundle $T_f\to Y$, there are canonical natural isomorphisms
  \[
    f_! \simeq f_\sharp \circ \Sigma^{-T_f}, \quad
    f^! \simeq \Sigma^{T_f} \circ f^*
  \]
  (see \cite[Theorem 6.18(2)]{Hoyois6}). In addition, for $X \in \Sch_B$ and $V\to X$ a vector bundle, there are canonical natural isomorphisms 
  \[
    f_\sharp\circ \Sigma^{\pm f^*V}\simeq \Sigma^{\pm V}\circ f_\sharp,\quad
    f^*\circ \Sigma^{\pm V}\simeq \Sigma^{\pm V}\circ f^*,
  \]
  the latter valid for an arbitrary morphism $f:Y\to X$ in $\Sm_B$. Moreover, for $f:Y\to X$ an arbitrary morphism in $\Sch_B$ and $V\to X$ a vector bundle, there is a canonical natural isomorphism
  \[
    f^!\circ\Sigma^{\pm V}\simeq \Sigma^{\pm f^*V}\circ f^!.
  \]
  (see the beginning of \cite[\S5.2]{Hoyois6}). Finally, for $f:Y\to X$ a regular embedding in $\Sch_B$ with normal bundle $N_f\to Y$, there is a canonical natural isomorphism
  \[
    f^!\simeq \Sigma^{-N_f}\circ f^*.
  \]

  In fact, the isomorphism $f^! \simeq \Sigma^{T_f} \circ f^*$ for smooth $f$ extends to the case of an lci morphism, as follows. For $f:Y\to X$ an lci morphism in 
  $\Sch_B$, we factor $f$ as $f=p\circ i$ with $i:Y\to Z$ a regular embedding and $p:Z\to X$ a smooth morphism (both in $\Sch_B$). This gives the relative virtual normal bundle $\nu_f:=[N_i]-[i^*T_{Z/X}]$ in $\sK(Y)$, independent up to canonical isomorphism on the choice of the factorization. Thus, we have the canonically defined suspension automorphism $\Sigma^{-\nu_f}$ and a canonical natural isomorphism
  \[
    f^!=i^!\circ p^!\simeq \Sigma^{-N_i}\circ i^*\circ \Sigma^{T_{Z/X}}\circ p^*\simeq
    \Sigma^{-\nu_f}\circ f^*.
  \]
  One can then construct a left adjoint $f_\sharp$ to $f^*$ by setting
  \[
    f_\sharp:=f_!\circ\Sigma^{\nu_f}.
  \]
  The functoriality $(gf)^*=f^*\circ g^*$ gives rise to the functoriality on the adjoints $(gf)_\sharp=g_\sharp\circ f_\sharp$ for composable lci morphisms.
\end{remark}
 
\begin{notation}
  Let $\pi : Z \to X$ be a morphism in $\Sch_B$. For $Y \in \H_\bul(Z)$, we set
  \[
    Y/X_\BM := \pi_!(\Sigma^\infty_\T Y) \in \SH(X),
  \]
  and for $Y \in \Sm_Z$, we make the abbreviation $Y/X_\BM$ for $Y_+/X_\BM = \pi_!(\Sigma^\infty_\T Y_+)$. In particular, we by definition have $Z/X_\BM = \pi_!(1_Z) \in \SH(X)$.

  Furthermore, for $i:W \inj Z$ the inclusion of a reduced closed subscheme, we set
  \[
    Z_W/X_\BM:=\pi_!(i_*(1_W)) .
  \]
  We observe two facts about this object. Firstly, letting $\pi'$ denote the composite $\pi \circ i : W\to X$, the isomorphism $\pi'_!=\pi_!\circ i_*$ determines an isomorphism $Z_W/X_\BM\simeq W/X_\BM$. Secondly, let $j:U:=Z\setminus W\to Z$ denote the open complement of $W$ and consider the localization distinguished triangle
  \[
    j_!j^!\to \id_{\SH(Z)}\to i_*i^*.
  \]
  Then the identities  $j^!= j^*$ and $i_*=i_!$ give a canonical distinguished triangle in $\SH(X)$,
  \[
    U/X_\BM\to Z/X_\BM\to Z_W/X_\BM\to U/X_\BM[1].
  \] 
\end{notation}

\begin{remark}
  \label{rmk:BMFunctor}
  The assignment $Z \mapsto Z/X_\BM$ extends to a functor
  \[
    (-)/X_\BM : (\Sch_X^\pr)^\op \to \SH(X).
  \]
  This is described in a number of places, for example \cite[\S 1]{LevISNC}; we recall the construction for the reader's convenience, referring to loc. cit. for details.

  Let $g : Z \to Y$ be a proper morphism in $\Sch_X$. Let $\pi_Y : Y \to X$ and $\pi_Z : Z \to X$ denote the structural morphisms. As mentioned in Notation~\ref{ntn:SH}, we have a natural isomorphism $\eta^g_{!*} : g_! \simeq g_*$. We may then define a natural transformation $g^* : \pi_{Y!} \to \pi_{Z!}g^*$ as the composition
  \[
    \pi_{Y!} \xrightarrow{u_g} \pi_{Y!}g_*g^* \xrightarrow{(\eta^g_{!*})^{-1}}
    \pi_{Y!}g_!g^* \simeq \pi_{Z!}g^*,
  \]
  where $u_g:\id_{\SH(Y)}\to g_*g^*$ is the unit of the adjunction. Evaluating this natural transformation at $1_Y$ gives a map $g^* = g/X_\BM : Y/X_\BM \to Z/X_\BM$, and it follows directly from the definitions that this construction satisfies $(gh)^*=h^*g^*$  for composable proper morphisms $g$, $h$. This establishes the claimed functoriality.
\end{remark}

\subsection{Duality for smooth projective schemes}

Let $(\sC, \otimes, 1, \mu, \tau)$ be a symmetric monoidal category. Recall that the the {\em dual} of an object $x$ of $\sC$ is a triple $(y, \delta, \epsilon)$ with $\delta_x:1\to x\otimes y$ and $\epsilon_x:y\otimes x\to 1$ maps such that the two compositions
\[
x\xr{\mu^{-1}}1\otimes x\xr{\delta_x\otimes\id_x}x\otimes y\otimes x\xr{\id_x\otimes\epsilon_x}x\otimes 1\xr{\mu}x
\]
and
\[
y\xr{\mu^{-1}} y\otimes 1\xr{\id_y\otimes\delta_x}y\otimes x\otimes y\xr{\epsilon_x\otimes \id_y}1\otimes y\xr{\mu} y
\]
are equal to the respective identity maps (this notion goes back to \cite{DoldPuppe}; see \cite{May} for details). In this subsection, we recall from \cite{Hoyois6} the construction of the dual of a smooth projective scheme in the stable motivic homotopy category.

\begin{remark}
  Let $(\sC, \otimes, 1, \mu, \tau)$ be a symmetric monoidal category and let $x \in \sC$. If a triple $(y,\delta,\epsilon)$ satisfying the above definition of the dual exists, then it is unique up to unique isomorphism. We often omit specific mention of $\delta$ and $\epsilon$ and denote the dual object $y$ by $x^\vee$.

  When $x$ admits a dual $(x^\vee,\delta,\epsilon)$, then it is immediate that $x^\vee$ is also dualizable, with dual $(x, \tau \circ \delta, \tau \circ \epsilon)$, so that $x$ is canonically isomorphic to $(x^\vee)^\vee$.

  Sending $x \mapsto x^\vee$ extends to a contravariant functor $(-)^\vee$ on the full subcategory of $\sC$ consisting of those objects $x$ that admit a dual with canonical isomorphism, and the above determines a natural isomorphism $((-)^\vee)^\vee \simeq \id$. For $f:x\to z$ a morphism of dualizable objects in $\sC$, the dual morphism $f^\vee:z^\vee\to x^\vee$ is the composition
  \[
    z^\vee\xr{\mu^{-1}}z^\vee\otimes 1\xr{\id_{z^\vee}\otimes \delta_x}z^\vee\otimes x\otimes x^\vee
    \xr{\id\otimes f\otimes\id}z^\vee\otimes z\otimes x^\vee\xr{\epsilon_z\otimes\id_{x^\vee}}1\otimes x^\vee\xr{\mu}x^\vee.
  \]
\end{remark}

\begin{lemma}
  \label{lem:DualProjFormula}
  Let $\pi_X : X \to B$ be an object of $\Sm_B$. View $X\times_BX$ as a $X$-scheme via the projection $p_2 : X\times_BX \to X$ onto the second factor. Then there are canonical isomorphisms
  \[
    \pi_X^*(X/B_\BM) \simeq    p_{2\sharp} \Sigma^{-p_1^* T_{X/B}}(1_{X\times_BX}) \simeq X\times_BX/X_\BM
  \]
  in $\SH(X)$ and a canonical isomorphism
  \[
    \pi_{X\sharp}(X\times_BX/X_\BM)\simeq X/B_\BM\wedge_BX/B
  \]
  in $\SH(B)$. 
\end{lemma}

\begin{proof}
  Consider the commutative square
  \[
    \xymatrix{
      X\times_BX\ar[r]^-{p_2}\ar[d]_{p_1}&X\ar[d]^{\pi_X}\\
      X\ar[r]_-{\pi_X}&B
    }
  \]
  This gives us the canonical isomorphism
  \[
    T_{X\times_BX/X}\simeq p_1^*T_{X/B}.
  \]
  We have the identities and canonical isomorphisms
  \begin{multline*}
    \pi_X^*(X/B_\BM)=\pi_X^*\pi_{X!}(1_X)\simeq \pi_X^*\pi_{X\sharp}\Sigma^{-T_{X/B}}(1_X)
    \iso^{\text{(base change)}} p_{2\sharp}p_1^*\Sigma^{-T_{X/B}}(1_X)\\\simeq p_{2\sharp}\Sigma^{-p_1^*T_{X/B}}(p_1^*1_X)=p_{2\sharp}\Sigma^{-p_1^*T_{X/B}}(1_{X\times_BX})
  \end{multline*}
  which gives us the first isomorphism in $\SH(X)$. The second follows from
  \[
    p_{2\sharp} \Sigma^{-p_1^* T_{X/B}}(1_{X\times_BX})\simeq
    p_{2\sharp} \Sigma^{- T_{X\times_BX/X}}(1_{X\times_BX})
    \simeq p_{2!}(1_{X\times_BX})=X\times_BX/X_\BM.
  \]
  Finally, to give the isomorphism in $\SH(B)$, we have
  \begin{multline*}
    \pi_{X\sharp}(X\times_BX/X_\BM)\simeq
    \pi_{X\sharp}(p_{2\sharp}(\Sigma^{-T_{X\times_BX/X}}(1_{X\times_BX}))\simeq \pi_{X\sharp}p_{1\sharp}(\Sigma^{-T_{X\times_BX/X}}(1_{X\times_BX}))\\\simeq
    \pi_{X\sharp}p_{1\sharp}(\Sigma^{-p_1^*T_{X/B}}(1_{X\times_BX}))\simeq
    \pi_{X\sharp}\Sigma^{-T_{X/B}}p_{1\sharp}(1_{X\times_BX})\simeq
    \pi_{X!}p_{1\sharp}p_2^*(1_X)\\\iso^{\text{(base change)}}
    \pi_{X!}\pi_X^*\pi_{X\sharp}(1_X)\simeq \pi_{X!}(1_X\wedge_X\pi_X^*(X/B))\\\iso^{\text{(projection formula)}}\pi_{X!}(1_X)\wedge_B (X/B)=X/B_\BM\wedge_B X/B.
  \end{multline*}
  The  base change isomorphisms follow from  \cite[Theorem 6.18(3)]{Hoyois6} and the projection formula is \cite[Theorem 6.18(7)]{Hoyois6}.
\end{proof}

\begin{construction}\label{constr:SWDual}
  Let $\pi_X : X \to B$ be an object of $\Sm_B$ that is proper (i.e. a smooth projective scheme over $B$). We recall (from Hoyois \cite{Hoyois6}, but see also the constructions of Riou \cite{RiouSWDual} and Ayoub \cite{Ayoub}) the construction of a duality between the objects $X/B$ and $X/B_\BM$ in $\SH(B)$.

  We first construct the coevaluation map $\delta_{X/B} : 1_B \to X/B\wedge_BX/B_\BM$. Applying the functoriality of $(-)/B_\BM$ from Remark~\ref{rmk:BMFunctor} to the proper map $\pi_X$ gives the map
  \[
    \pi_X^* : 1_B = B/B_\BM \to X/B_\BM = \pi_{X!}(1_X) \simeq \pi_{X\sharp}(\Sigma^{-T_{X/B}}(1_X))
  \]
  in $\SH(B)$, and the diagonal $\Delta_{X/B} : X \to X\times_BX$ induces via the functoriality of  $(-)/X:\Sm_X\to \SH(X)$   the map
  \[
    \Delta_{X/B*}:=\Delta_{X/B}/X : 1_X = X/X \to X\times_BX/X
  \]
  in $\SH(X)$. We may put together these two maps to obtain the composition
  \[
    1_B \xr{\pi_X^*} \pi_{X\sharp}(\Sigma^{-T_{X/B}}(1_X)) \xr{\pi_{X\sharp}\Sigma^{-T_{X/B}}(\Delta_{X/B*})} \pi_{X\sharp}\Sigma^{-T_{X/B}} (X\times_BX/X),
  \]
  Using the identification $\pi_{X\sharp}\Sigma^{-T_{X/B}}(X\times_BX/X) \simeq X/B\wedge_BX/B_\BM$ from Lemma~\ref{lem:DualProjFormula}, this gives the desired map $\delta_{X/B}$.

  We now construct the evaluation map $\epsilon_{X/B} : X/B_\BM \wedge_B X/B \to 1_B$ (which does not require properness of $\pi_X$). Here we apply the functoriality of $(-)/X_\BM$ to the proper map $\Delta_{X/B}$ to obtain the map
  \[
    \Delta_{X/B}^* : X\times_BX/X_\BM \to X/X_\BM = 1_X
  \]
  in $\SH(X)$, and the functoriality of $(-)/B$ to the map $\pi_X$ to obtain the map
  \[
    \pi_{X*} : X/B \to B/B = 1_B.
  \]
  Putting these together yields the composition
  \[
    \pi_{X\sharp}(X\times_BX/X_\BM)
    \xr{\pi_{X/B\sharp}(\Delta^*)} \pi_{X/B\sharp}(1_X) =X/B \xr{\pi_{X*}} 1_B.
  \]
  Now using the identification $X/B_\BM\wedge_BX/B \simeq
  \pi_{X\sharp}(X\times_BX/X_\BM)$ from Lemma~\ref{lem:DualProjFormula}, we get the desired map $\epsilon_{X/B}$.

  It is shown in \cite[Corollary 6.13]{Hoyois6} that the triple $(X/B_\BM, \delta_{X/B}, \epsilon_{X/B})$ is the dual of $X/B$ in $\SH(B)$.  Using our notation for Thom spaces, we have
  \[
    X/B_\BM=\pi_{X!}(1_X)\simeq \pi_{X\#}\circ \Sigma^{-T_{X/B}}(1_X)=\Th_X(-T_{X/B})/B
  \]
  the coevaluation map is
  \[
    \delta_{X/B}:1_B\to X/B\wedge_B\Th_X(-T_{X/B})/B
  \]
  and the evaluation map is
  \[
    \epsilon_{X/B}:\Th_X(-T_{X/B})/B\wedge_BX/B\to  1_B
  \]
\end{construction}

The dualizability of $X/B$ as above allows one to define an Euler characteristic for smooth projective schemes:

\begin{definition}
  For $X \to B$ in $\Sm_B$ and proper, the {\em Euler characteristic}  $\chi(X/B)\in \End_{\SH(B)}(1_B)$ is the composition
  \[
    1_B\xr{\delta_{X/B}}X/B\wedge_B\Th_X(-T_{X/B})/B\xr{\tau}\Th_X(-T_{X/B})/B\wedge_BX/B\xr{\epsilon_{X/B}} 1_B,
  \]
  where $\tau$ denotes the symmetry isomorphism, and $\delta_{X/B}$ and $\epsilon_{X/B}$ are as in Construction~\ref{constr:SWDual}.
\end{definition}
 
To finish this section, we give an alternative characterization of the Euler characteristic $\chi(X/B)$ just defined (Lemma~\ref{lem:simp} below). We proved this result in the case $B=\Spec k$ for $k$ a field in \cite{LevEnum}; the proof in this more general setting is exactly the same.

\begin{construction}\label{const:beta}
We consider $X\times_BX$ as a smooth $X$-scheme via the projection $p_2$. The diagonal 
 $\Delta_X:X\to X\times_BX$ gives a section to $p_2$. Let $q:X\times_BX\to X\times_BX/(X\times_BX\setminus \Delta_X(X))$ be the 	quotient map. We have the Morel-Voevodsky purity isomorphism \cite[Theorem 3.2.23]{MorelVoev}, which gives the isomorphism in $\H_\bul(X)$
  \[
 mv_\Delta:X\times_BX/(X\times_BX\setminus \Delta_X(X))\xr{\sim}\Th_X(N_{\Delta_X})
  \]
 Composing the 0-section $s_N:X\to N_{\Delta_X} $ with the quotient map $N_{\Delta_X}\to \Th(N_{\Delta_X})$  defines $\bar{s}_N:X\to \Th_X(N_{\Delta_X})$. 
It follows from the proof of the purity isomorphism that we have the commutative diagram in $\H_\bul(X)$
 \begin{equation}\label{eqn:beta}
 \xymatrix{
 X\times_BX_+\ar[r]^-q& X\times_BX/(X\times_BX\setminus \Delta_X(X))\ar[r]^-{mv_\Delta}&\Th_X(N_{\Delta_X})\\
 &X_+\ar[u]^{q\circ \Delta_X}\ar[ur]_{\bar{s}_N}\ar[ul]^{\Delta_X}
 }
 \end{equation}
 Finally, we have  the isomorphism $N_{\Delta_X}\simeq T_{X/B}$ of vector bundles over $X$  furnished by the composition
\[
    T_{X/B}\xr{\sim}\Delta_X^*p_2^*T_{X/B}\xr{i_2}\Delta_X^*(p_1^*T_{X/B}\oplus p_2^*T_{X/B})\simeq
    \Delta_X^*T_{X\times_BX/B}\xr{\pi}N_{\Delta_X}
\]  
where $\pi: \Delta_X^*T_{X\times_BX/B}\to N_{\Delta_X}$ is the canonical projection. Putting these maps together gives us the composition in $\H_\bul(X)$
\[
X_+\xr{\Delta_X}X\times_BX\xr{q}X\times_BX/(X\times_BX\setminus \Delta_X(X))\xr{mv_\Delta}
\Th_X(N_{\Delta_X})\xr{\sim} \Th_X(T_{X/B});
\]
the commutativity of the diagram \eqref{const:beta} shows that this composition is equal to the map $\bar{s}_{T_{X/B}}:X_+\to \Th_X(T_{X/B})$, induced as for $\bar{s}_{N_{\Delta_X}}$ by the zero-section $s_{T_{X/B}}:X\to T_{X/B}$.

Applying $\Sigma^{-T_{X/B}}\Sigma^\infty_{\P^1}(-)$ and the isomorphism
\[
\Sigma^{-T_{X/B}}\Th(T_{X/B})\simeq \Sigma^{-T_{X/B}}\circ\Sigma^{T_{X/B}}(1_X)\simeq 1_X
\]
 gives the  morphism
\[
\tilde{\beta}_{X/B}:\Sigma^{-T_{X/B}}(1_X)\to 1_X
\]
in $\SH(X)$. Finally, applying $\pi_{X\#}$ gives the morphism
\[
\beta_{X/B}:\Th(-T_{X/B})\to X/B
\]
in $\SH(B)$.

Summarizing our construction, $\beta_{X/B}$ is given by applying $\pi_{X\#}$ to the composition in $\SH(X)$
\begin{multline}\label{const:betaComp1}
\Sigma^{-T_{X/B}}(1_X)\xr{\Sigma^{-T_{X/B}}(\Delta_X)}
\Sigma^{-T_{X/B}}(X\times_BX/X)\xr{\Sigma^{-T_{X/B}}(q)}\\
\Sigma^{-T_{X/B}}([X\times_BX/(X\times_BX\setminus \Delta_X)]/X)
\xr{\Sigma^{-T_{X/B}}(mv_{\Delta_X})}\\
\Sigma^{-T_{X/B}}(\Th_X(N_{\Delta_X}))\simeq
\Sigma^{-T_{X/B}}(\Th_X(T_{X/B}))\\\simeq
\Sigma^{-T_{X/B}}\circ\Sigma^{T_{X/B}}(1_X)\simeq
1_X
\end{multline}
or equivalently, to the composition
\begin{equation}\label{const:betaComp2}
\Sigma^{-T_{X/B}}(1_X)\xr{\Sigma^{-T_{X/B}}(\bar{s}_{T_{X/B}})}
\Sigma^{-T_{X/B}}(\Th_X(T_{X/B}))\simeq
\Sigma^{-T_{X/B}}\circ\Sigma^{T_{X/B}}(1_X)\simeq
1_X
\end{equation}

\end{construction}

\begin{lemma}\label{lem:simp} For $X$ smooth and proper over $B$, $\chi(X/B)$ is equal to the composition
\[
1_B \xrightarrow{\pi_X^*}\Th(-T_{X/B})
\xrightarrow{\beta_{X/B}}X/B\xrightarrow{
\pi_{X*}}1_B
\]
\end{lemma}

\begin{proof} Let $p_i:X\times_BX\to X$, $i=1,2$, be the projections. The map 
\[
\ev_X:X/B^\vee\wedge_BX/B\to 1_B
\]
is the composition
\begin{multline*}
X/B^\vee\wedge_B X/B=
\Th(-T_{X/k})\wedge_BX/B\\=
\Th(-p_1^*T_{X/k})\xr{q} \Th(-p_1^*T_{X/B})/  \Th(-j^*p_1^*T_{X/B})\simeq 
\Th( -T_{X/B}\oplus N_{\Delta_X})\\\simeq 
 \Th(-T_{X/B}\oplus T_{X/B})\simeq X/B\xr{\pi_{X}} 1_B.
 \end{multline*}
From our description of $\beta_{X/B}$ as $\pi_{X\#}$ applied to the composition \eqref{const:betaComp1}, we see that $\pi_{X*}\circ \beta_{X/B}=\ev_{X/B}\circ \Th(\Delta_X)$. Also, $\pi_X^*=\pi_X^\vee$ and $\pi_X^\vee$ is given by
\[
1_B\xr{\delta_{X/B}}X/B\wedge_B X/B^\vee
=\Th(-p_2^*T_{X/B})\xr{p_2}\Th(-T_{X/B})
\]

It follows from the construction of the map $\delta_{X/B}$  described above that 
\[
\delta_{X/B}=\Th(\Delta_X)\circ \pi_X^*.
\]

This gives us the commutative diagram
\[
\xymatrixcolsep{40pt}
\xymatrix{
&X/B\wedge_BX/B^\vee\ar[r]^{\tau_{X, X^\vee}}\ar@{=}[d]&X/B^\vee\wedge_BX/B\ar[dr]^{\ev_X}\\
1_B\ar[ur]^-{\delta_{X/B}}\ar[dr]_{\pi_X^*}&
\Th(-p_2^*T_{X/B})\ar[r]^{\tau_{X, X^\vee}}&\Th(-p_1^*T_{X/B})\ar@{=}[u]
&1_B\\
&\Th(-T_{X/B})\ar[ur]_-{\Th(\Delta_X)}\ar[u]_{\Th(\Delta_X)} \ar[r]_{\beta_{X/B}}
&X/B\ar[ur]_{\pi_{X*}} 
}
\]
\end{proof}

\section{$\SL$-oriented motivic spectra}
\label{sec:SL}

In this section, we discuss the definition and basic features of $\SL$-oriented motivic spectra. This is treated by Ananyevskiy in \cite{AnanSL} in the context of the motivic stable homotopy category $\SH(k)$ for $k$ a field. Essentially all of the constructions in op. cit. go through without change in the setting of a separated noetherian base-scheme $B$ of finite Krull dimension; the most one needs to do is replace a few arguments that rely on Jouanolou covers with some properties coming out of the six-functor formalism. We will recall and suitably extend Ananyevskiy's treatment here without any claim of originality.

\begin{definition}
  A \emph{motivic commutative ring spectrum} in $\SH(B)$ is a triple $(\sE, \mu, u)$ with $\sE$ in $\SH(B)$, and $\mu:\sE\wedge_B\sE\to \sE$, $u:1_B\to \sE$ morphisms in $\SH(B)$, defining a commutative monoid object in the symmetric monoidal category $\SH(B)$. We usually drop the explicit mention of the multiplication $\mu$ and unit $u$ unless these are needed.
\end{definition}

\begin{notation}
  Recall that, given a motivic spectrum $\sE \in \SH(B)$, we have a notion of \emph{$\sE$-cohomology} $\sE^{**}(X)$ for $X \in \H(B)$: this is the bigraded abelian group defined by
  \[
    \sE^{a,b}(X) := \Hom_{\SH(B)}(\Sigma^\infty_\T X_+, \rS^{a,b} \wedge \sE)
  \]
  for $a,b \in \Z$, where $\rS^{a,b} \in \SH(B)$ denotes the usual bigraded stable motivic sphere. This of course specializes to give the $\sE$-cohomology of objects $X \in \Sm_B$. It also specializes to give \emph{$\sE$-cohomology with supports} $\sE^{**}_Z(X)$ for $X \in \Sm_B$ and $Z \subseteq X$ a closed subset: namely, we define
  \[
    \sE^{a,b}_Z(X):=\sE^{a,b}(X/(X\setminus Z)).
  \]
  For instance, if $V\to X$ is a vector bundle, then $\sE^{a,b}_{0_V}(V)=\sE^{a,b}(\Th(V))$, where $0_V\subset V$ is the image of the zero-section. We further define
  \[
    \sE^{a,b}_Z(\Th(V)) := \sE^{a,b}_{0_V \cap q^{-1}(Z)}(V)
  \]
  for $Z \subseteq X$ a closed subset and $q : V \to X$ a vector bundle.

  A commutative ring spectrum structure on $\sE$ determines a natural cup product structure on $\sE$-cohomology and $\sE$-cohomology with supports in the usual manner; we refrain from spelling it out in detail here.
\end{notation}

\begin{notation}
  For a scheme $X$ and a rank $r$ vector bundle $V\to X$, we let $\det V\to X$ denote the determinant line bundle, defined by $\det V := \Lambda^rV$. Given two vector bundles $V_1\to X$ and $V_2\to X$, we have a canonical isomorphism
  \[
    \alpha_{V_1, V_2}:\det(V_1\oplus V_2)\to (\det V_1) \otimes_{\sO_X} (\det V_2),
  \]
  characterized by requiring that, for a local basis of sections $s_1^1,\ldots, s_n^1$ of $V_1$ and $s_1^2,\ldots, s_m^2$ of $V_2$, we have
  \[
    \alpha_{V_1, V_2}(((s_1^1,0)\wedge\ldots \wedge(s_n^1,0)\wedge(0, s_1^2)\wedge\ldots\wedge(0, s_m^2)))= (s_1^1\wedge\ldots\wedge s_n^1)\otimes(s_1^2\wedge\ldots\wedge s_m^2).
  \]
  This extends to a canonical natural isomorphism
  \[
    \alpha_E:\det V \to (\det V_1) \otimes_{\sO_X} (\det V_2)
  \]
  for each exact sequence $E$ of vector bundles on $X$,
  \[
    0\to V_1\to V\to V_2\to 0;
  \]
  one can define $\alpha_E$ by choosing local splittings and noting that the resulting isomorphism is independent of the choice of splitting.
\end{notation}

\begin{definition}
  \label{def:SLOr}
  An {\em $\SL$-orientation} of a motivic commutative ring spectrum $\sE$ in $\SH(B)$ is an assignment of elements $\th_{V,\theta}\in \sE^{2r,r}(\Th(V))$ for each pair $(V,\theta)$ consisting of a rank $r$ vector bundle $V\to X$ (for any $r \ge 0$) with $X\in \Sm_B$ and  an isomorphism $\theta:\det V\to \sO_X$ of line bundles, satisfying the following axioms:
  \begin{enumerate}
  \item \emph{Naturality}: Let $(V \to X, \theta : \det V \to \sO_X)$ be as above and let $f:Y\to X$ be a morphism in $\Sm_B$. Consider the vector bundle $f^*V\to Y$ and isomorphism $f^*\theta:\det f^*V\simeq f^*\det V\to \sO_Y$. Then we have $f^*(\th_{V,\theta})=\th_{f^*V, f^*\theta}$.
  \item \emph{Products}: Let $(V_1\to X,\theta_1:\det V_1 \to \sO_X)$ and $(V_2\to X,\theta_2:\det V_2 \to \sO_X)$ be two pairs as above. Consider the vector bundle $V_1 \oplus V_2 \to X$ and the isomorphism $\theta_1\wedge\theta_2:\det(V_1\oplus V_2)\to \sO_X$ defined by
    \[
      \theta_1\wedge\theta_2 := (\theta_1\otimes\theta_2)\circ\alpha_{V_1, V_2}.
    \]
    Then $\th_{V_1\oplus V_2, \theta_1\wedge\theta_2}=\th_{V_1,\theta_1}\cup\th_{V_2,\theta_2}$.
  \item \emph{Normalization}: For $X \in \Sm_B$, consider the trivial vector bundle $V=\sO_X$ and the identity isomorphism $\theta : V \to \sO_X$. Then, under the canonical identification $\Th(V) \simeq \Sigma_\T X_+$, the element $\th_{V,\theta}\in \sE^{2,1}(\Th(V))$ is the image of the unit $1 \in \sE^{0,0}(B)$ under the composition
    \[
      \sE^{0,0}(B)\xrightarrow{\pi_X^*}\sE^{0,0}(X)\iso^{\text{suspension}}\sE^{2,1}(\Sigma_\T X_+).
    \]
  \end{enumerate}

  An {\em $\SL$-oriented motivic spectrum} is a pair $(\sE, \th_{(-)})$ with $\sE$ a motivic commutative ring spectrum and $\th_{(-)}$ an $\SL$-orientation on $\sE$.
\end{definition}

\begin{variant}
  A {\em $\GL$-orientation}, or simply {\em orientation}, on a motivic commutative ring spectrum $\sE$ is an assignment $(V\to X)\mapsto \th_V\in \sE^{2r, r}(\Th(V))$, where $V\to X$ is a rank $r$ vector bundle (for any $r \ge 0$) on $X\in \Sm_B$, satisfying the evident modifications of the axioms (i)--(iii) in Definition~\ref{def:SLOr}, i.e. omitting the conditions on the determinant line bundle. The pair $(\sE, \th_{(-)})$ is   called a {\em $\GL$-oriented motivic spectrum}, or more simply, an \emph{oriented motivic spectrum}.
\end{variant}

For the remainder of the section, we fix an $\SL$-oriented motivic spectrum $\sE$. Let us first observe that the $\SL$-orientation determines Thom isomorphisms in $\sE$-cohomology for oriented vector bundles:

\begin{lemma}
  \label{lem:ThomIso1}
  Let $(\sE, \th_{(-)})$ be an $\SL$-oriented motivic spectrum and let $q:V\to X$ be a rank $r$ vector bundle on $X \in \Sm_B$ with an isomorphism $\theta:\det V\to \sO_X$. Then sending $x\in \sE^{a, b}_Z(X)$ to  $q^*(x)\cup\th_{V,\theta}$ defines an isomorphism
  \[
    \vartheta_{V,\theta}:\sE^{a, b}_Z(X)\to \sE^{a+2r, b+r}_Z(\Th(V))
  \]
  natural in $(X, V, \theta)$. 
\end{lemma}

\begin{proof}
  The naturality of the maps $\vartheta_{V,\theta}$ follows from the naturality of the Thom classes, i.e. property (i) in their definition. 

  It follows from properties (i)--(iii) of the Thom class that for $V=\oplus_{i=1}^r\sO_Xe_i$ and $\theta:\det V\to \sO_X$ the canonical isomorphism given by $\theta(e_1\wedge\ldots\wedge e_n)=1$, the map $\vartheta_{V,\theta}$  is the suspension isomorphism
  \[
    \sE^{a, b}_Z(X)\simeq \sE^{a+2r, b+r}(\Sigma^r_TX_+/\Sigma^r_T(X\setminus Z)_+)\simeq
    \sE^{a+2r, b+r}_Z(\Th(V)).
  \]
  The naturality of the maps $\vartheta_{V,\theta}$ allow one to use a Mayer-Vietoris sequence for a trivializing open cover of $X$ for $V$ to show that $\vartheta_{V,\theta}$  is an isomorphism in general.
\end{proof}

The above Thom isomorphism may be extended to vector bundles without orientation by introducing \emph{twists} to $\sE$-cohomology, as follows.

\begin{definition}
  \label{def:Twist}
  For $X\in \Sm_B$ and $q:L\to X$ a line bundle, we define the \emph{twisted $\sE$-cohomology} $\sE^{**}(X;L)$ by
  \[
    \sE^{a,b}(X;L):=\sE^{a+2, b+1}(\Th(L))=\sE^{a+2, b+1}_{0_L}(L).
  \]
  We also have a version with supports: given in addition a closed subset $Z \subseteq X$, we define
  \[
    \sE^{a,b}_Z(X;L):=\sE^{a+2, b+1}_{0_L\cap q^{-1}(Z)}(L).
  \]
  Finally, when we also have a vector bundle $V \to X$, we similarly define
  \[
    \sE^{a,b}(\Th(V);L) := \sE^{a,b}_{0_V}(V;q^*L), \quad
    \sE^{a,b}_Z(\Th(V);L) := \sE^{a,b}_{0_V\cap q^{-1}(Z)}(V;q^*L);
  \]
  these definitions may be rewritten a bit, e.g. for the former we have
  \[
    \sE^{a,b}(\Th(V);L) = \sE^{a,b}_{0_V}(V;q^*L) = \sE^{a+2,b+1}_{0_{V \oplus L}}(V\oplus L) = \sE^{a+2,b+1}(\Th(V\oplus L)).
  \]
\end{definition}

\begin{remark}
  The product structure on $\sE$-cohomology extends to a product structure on twisted $\sE$-cohomology: namely, for $X \in \Sm_B$ and $L,M$ two line bundles on $X$, combining the cup product
  \[
    \cup:\sE^{a+2,b+1}_{0_L}(L)\otimes \sE^{c+2,d+1}_{0_M}(M)\to \sE^{a+c+4, b+d+2}_{0_{L\oplus M}}(L\oplus M),
  \]
  the canonical isomorphism $\alpha_{L, M}:\det(L\oplus M)\to L\otimes M$, and the Thom isomorphism $\vartheta_{L\oplus M}$, we obtain a map
  \[
    \cup:\sE^{a,b}(X;L)\otimes \sE^{c,d}(X;M)\to \sE^{a+c, b+d}(X;L\otimes M),
  \]
  as well as a version with supports. 
\end{remark}

\begin{remark}
  \label{rmk:TwistGeneral}
  The definitions of twisted $\sE$-cohomology given in Definition~\ref{def:Twist} are instances of a more general definition. Namely, for $X \in \Sm_B$, $L \to X$ a line bundle, and any $T \in \SH(B)$ equipped with an identification $T \simeq \pi_{X\sharp}(T')$ for some $T' \in \SH(X)$, we may define the twisted $\sE$-cohomology $\sE^{**}(T;L)$ by
  \[
    \sE^{a,b}(T;L) := \Hom_{\SH(X)}(\Sigma^L T', \rS^{a+2,b+1} \wedge \pi_X^*\sE).
  \]
  Of course, this notation is abusive since $\sE^{**}(T;L)$ really depends on $T' \in \SH(X)$. This construction recovers the notions introduced in Definition~\ref{def:Twist} as follows:
  \begin{itemize}
  \item We have $\sE^{**}(X;L) \simeq \sE^{**}(T;L)$ for $T = \Sigma^\infty_\T X_+ \simeq \pi_{X\sharp}(1_X)$.
  \item We have $\sE^{**}_Z(X;L) \simeq \sE^{**}(T;L)$ for $T = X_Z/B \simeq \pi_{X\sharp}(X_Z/X)$.
  \item We have $\sE^{**}(\Th(V);L) \simeq \sE^{**}(T;L)$ for $T = \Th(V) \simeq \pi_{X\sharp}(\Th_X(V))$.
  \end{itemize}
  The extra generality will be invoked later on (see Lemma~\ref{lem:Duality}) for the case
  \[
    T = X_Z/B_\BM = \pi_{X!}(X_Z/X_\BM) \iso \pi_{X\sharp}(\Sigma^{-T_{X/B}} X_Z/X_\BM),
  \]
  where the last identification is by Atiyah duality (Remark~\ref{rmk:Atiyah}).
\end{remark}

The following construction is due to Ananyevskiy  \cite[Corollary 1]{AnanSL}.

\begin{construction}
  \label{const:CanonThom}
  Let $X \in \Sm_B$, let $q:V\to X$ be a rank $r$ vector bundle, and let $p:L\to X$ denote the determinant bundle $\det V$; let $p':L^{-1}\to X$ denote the inverse of $L$. Then we have canonical isomorphisms
  \[
    \alpha_{V, L^{-1}}:\det(V\oplus L^{-1})\to \sO_X,\quad
    \alpha_{L^{-1}\oplus L}:\det(L^{-1}\oplus L)\to \sO_X.
  \]
  Consider now the two pullback bundles
  \[
    p_V:q^*(L^{-1}\oplus L)\to V, \quad q_L\oplus p'_L:p^*(V\oplus L^{-1})\to L,
  \]
  which inherit trivializations of their determinants. For $(\sE, \th_{(-)})$ an $\SL$-oriented motivic spectrum, Lemma~\ref{lem:ThomIso1} gives us isomorphisms
  \[
    \vartheta_{q^*(L^{-1}\oplus L), q^*\alpha_{L^{-1}\oplus L}}:\sE^{a+2r,b+r}_{0_V}(V)\to
    \sE^{a+2r+2,b+r+1}_{q^{-1}(0_{L^{-1}}\oplus 0_L)\cap p_V^{-1}(0_V)}(q^*(L^{-1}\oplus L))
  \]
  and
  \[
    \vartheta_{p^*(V\oplus L^{-1}), p^*\alpha_{V\oplus L^{-1}}}:\sE^{a+2,b+1}_{0_L}(L)\to
    \sE^{a+2r+2,b+r+1}_{p^{-1}(0_V\oplus 0_{L^{-1}})\cap p_L^{\prime-1}(0_V)}(p^*(V\oplus L^{-1})).
  \]
  However, as $X$-schemes $q^*(L^{-1}\oplus L)$ and $p^*(V\oplus L^{-1})$ are both canonically isomorphic to $V\oplus L^{-1}\oplus L\to X$, and via this isomorphism, the closed subsets 
  $q^{-1}(0_{L^{-1}}\oplus 0_L)\cap p_V^{-1}(0_V)$ and $p^{-1}(0_V\oplus 0_{L^{-1}})\cap p_L^{\prime-1}(0_V)$ are both equal to $0_{V\oplus L^{-1}\oplus L}$. We thus have a canonical isomorphism 
  \[
    \phi: \sE^{a+2r+2,b+r+1}_{q^{-1}(0_{L^{-1}}\oplus 0_L)\cap p_V^{-1}(0_V)}(q^*(L^{-1}\oplus L))
    \to \sE^{a+2r+2,b+r+1}_{p^{-1}(0_V\oplus 0_{L^{-1}})\cap p_L^{\prime-1}(0_V)}(p^*(V\oplus L^{-1}))
  \]
  
  Combining all of the above the, we obtain the \emph{Thom isomorphism}
  \begin{equation}\label{eqn:CanonThom}
    \vartheta_V:\sE^{a,b}(X;L)\to \sE^{a+2r, b+r}(\Th(V)),
  \end{equation}
  defined by
  \[
    \vartheta_V:=\vartheta_{q^*(L^{-1}\oplus L), q^*\alpha_{L^{-1}\oplus L}}^{-1}\circ\phi\circ \vartheta_{p^*(V\oplus L^{-1}), p^*\alpha_{V\oplus L^{-1}}}.
  \]
  This construction extends to cohomology with supports in an evident manner.   
\end{construction}

We next discuss the Thom classes in twisted $\sE$-cohomology governing this more general Thom isomorphism.

\begin{definition}
  \label{def:CanonThomClass}
  Let $X \in \Sm_B$ and let  $q:V\to X$ be a rank $r$ vector bundle. The \emph{canonical Thom class} $\th_V \in \sE^{2r, r}(\Th(V);\det^{-1} V)$ is defined as follows. As noted in Definition~\ref{def:Twist}, we have
  \[
    \sE^{2r, r}(\Th(V);\det^{-1} V) = \sE^{2r+2, r+1}(\Th(V\oplus \det^{-1}V)).
  \]
  The isomorphism
  \[
    \alpha := \alpha_{V,\det^{-1}V}:\det(V\oplus \det^{-1}V)\to \sO_X
  \]
  gives us the Thom class $\th_{V\oplus \det^{-1}V, \alpha}\in \sE^{2r+2, r+1}(\Th(V\oplus \det^{-1}V))$, and we define $\th_V$ to be the corresponding element of $\sE^{2r, r}(\Th(V);\det^{-1} V)$ under the above identification.
\end{definition}

\begin{remark}
  Let $X$ and $V$ be as in Definition~\ref{def:CanonThomClass}. Let $\pi_X:X\to B$ denote the structure morphism. Then the Thom class $\th_V$ is an element of 
  \begin{align*}
    \sE^{2r, r}(\Th(V);\det^{-1} V)
    &=\sE^{2r+2, r+1}(\pi_{X\sharp}(\Sigma^{V\oplus \det^{-1}(V)}(1_X))\\
    &=  \Hom_{\SH(B)}(\pi_{X\sharp}(\Sigma^{V\oplus \det^{-1}(V)}(1_X), \rS^{2r+2, r+1}\wedge \sE) \\
    &\simeq \Hom_{\SH(X)}(\Sigma^{V\oplus \det^{-1}(V)}(1_X), \rS^{2r+2, r+1}\wedge \pi_X^*\sE). \\
    &\simeq \Hom_{\SH(X)}(1_X, \rS^{2r+2, r+1}\wedge \Sigma^{-(V\oplus \det^{-1}(V))} \pi_X^*\sE).
  \end{align*}
  Thus, via the multiplication on $\sE$, the class $\th_V$ induces a map
  \[
    \times\th_V :\pi_X^*\sE\to \rS^{2r+2, r+1}\wedge \Sigma^{-(V\oplus \det^{-1}(V))} \pi_X^*\sE
  \]
  in $\SH(X)$.
\end{remark}

\begin{lemma}
  \label{lem:GenThomIso}
  The map $\times\th_V$ defined above is an isomorphism in $\SH(X)$.
\end{lemma}

\begin{proof}
  For each object $x\in \SH(X)$, $\times\th_V$ induces a map
  \[
    \vartheta_{V,x}:\Hom_{\SH(X)}(x, \pi_X^*\sE)\to \Hom_{\SH(X)}(x,\rS^{2r+2, r+1}\wedge \Sigma^{-(V\oplus \det^{-1}(V))} \pi_X^*\sE)
  \]
  By the Yoneda lemma, it suffices to show that $\vartheta_{V,x}$ is an isomorphism for all $x\in \SH(X)$. 
  
  The collection of objects $x$ for which $\vartheta_{V,x}$ is an isomorphism is closed under arbitrary direct sums, hence is a localizing subcategory of $\SH(X)$. The objects $x=\rS^{a,b} \wedge 1_X$ are contained in this subcategory, since for these objects the map $\vartheta_{V,x}$ identifies with the Thom isomorphism $\theta_{V\oplus\det^{-1}(V),\alpha}$ (where $\alpha$ is as in Definition~\ref{def:CanonThomClass}). For $p:Y\to X$ in $\Sm_X$, applying $p^*$ and using the adjunction with $p_\sharp$ shows that furthermore $\vartheta_{V,x}$ is an isomorphism for $x=\rS^{a,b} \wedge Y/X$. As $\SH(X)$ is generated as a localizing subcategory by the objects $\rS^{a,b} \wedge Y/X$, this proves the lemma.
\end{proof}

\begin{definition}
  For $X\in \Sm_B$ and $V\to X$ a rank $r$ vector bundle, we define
  \[
    \vartheta^\sE_V:\Sigma^{1-\det V}\pi_X^*\sE\to \Sigma^{r-V}\pi_X^*\sE
  \]
  to be the composition of isomorphisms
  \[
    \Sigma^{1-\det V}\pi_X^*\sE\xr{\Sigma^{\det V-1}(\times\th_{\det V\oplus \det^{-1}V})^{-1}}
    \Sigma^{\det^{-1}V-1}\pi_X^*\sE\xr{\Sigma^{1-\det^{-1}V}(\times\th_V)}
    \Sigma^{r-V}\pi_X^*\sE.
  \]
\end{definition}

\begin{remark}
  \label{rem:CanonicalThom}
  Let $X \in \Sm_B$ and let $V\to X$ be a rank $r$ vector bundle. Under the identification $\Th(V)=\pi_{X\#}\Sigma^V(1_X)$ and the isomorphisms
  \[
    \sE^{a,b}(X,\det V)\simeq \Hom_{\SH(X)}(1_X, \rS^{a+2,b+1}\wedge \Sigma^{-\det V}\pi_X^*\sE)
  \]
  and
  \[
    \sE^{2r+a, r+b}(\Th(V))\simeq \Hom_{\SH(X)}(1_X, \rS^{2r+a, r+b} \wedge \Sigma^{-V}\pi_X^*\sE)
  \]
  the Thom isomorphism $\vartheta_V:\sE^{a,b}(X,\det V)\to \sE^{2r+a, r+b}(\Th(V))$ is the map induced by $\vartheta_V^\sE$.
\end{remark}

\begin{remark}
  \label{rem:MultThom}
  Let $X \in \Sm_B$ and let $V\to X$ and $W\to X$ be vector bundles of respective ranks $r_V$ and $r_W$. We have have the multiplication map
  \[
    \sE^{a,b}(\Th(V), \det^{-1}V)\otimes\sE^{c,d}(\Th(W),\det^{-1}W)\to
    \sE^{a+c, b+d}(\Th(V\oplus W),\det^{-1}(V\oplus W))
  \]
  induced by the isomorphism
  \[
    V/(V\setminus\{0_V\})\wedge_XW/(W\setminus\{0_W\})\to (V\oplus W)/(V\oplus W\setminus\{0_{V\oplus W}\})
  \]
  in $\H_\bul(X)$ and our canonical isomorphism $\det(V)\otimes\det(W)\simeq \det(V\oplus W)$. The multiplicative property of the Thom classes (Definition~\ref{def:SLOr}(ii)) implies a similar multiplicativity for the canonical Thom classes:
  \[
    \th_{V\oplus W}=\th_V\cup \th_W.
  \] 

  This then implies, roughly speaking, that
  \[
    \text{``}\ \vartheta_{W\oplus V}^\sE=\vartheta_W^\sE\circ \vartheta_V^\sE.\ \text{''}
  \]
  More precisely, after using properties of $\Sigma^{(-)}$ to make the necessary identifications, the following diagram commutes:
  \[
    \xymatrix{
      \Sigma^{1-\det(W\oplus V)}\pi_X^*\sE\ar[rr]^{\vartheta_{\det W\oplus \det V}^\sE}\ar[d]_{\vartheta_{W\oplus V}^\sE}&&\Sigma^{2-(\det W\oplus\det V)}\pi_X^*\sE\ar[r]^\sim&
      \Sigma^{1-\det W}\Sigma^{1-\det V}\pi_X^*\sE\ar[d]^{\Sigma^{1-\det W}\vartheta_E^\sE}\\
      \Sigma^{r_W+r_V-(W\oplus V)}\pi_X^*\sE&&\Sigma^{r_V-V}\Sigma^{1-\det W}\pi_X^*\sE\ar[ll]^{\Sigma^{r_V-V}\vartheta^\sE_W}&\Sigma^{1-\det W}\Sigma^{r_V-V}\pi_X^*\sE\ar[l]^\sim
    }
  \]
\end{remark}

\begin{remark}
  Using Remark~\ref{rem:MultThom}, the definition of $\vartheta^\sE_{-}$ extends to virtual bundles by setting $\vartheta^\sE_{V-W}:=\vartheta_V^\sE\circ (\vartheta_W^\sE)^{-1} =
  (\vartheta_W^\sE)^{-1}\circ \vartheta_V^\sE$ (with these identities understood as in Remark~\ref{rem:MultThom}), giving the isomorphism
  \[
    \vartheta^\sE_{V-W}:\Sigma^{1-\det V\otimes\det^{-1} W}\pi_X^*\sE\xr{\sim}\Sigma^{r_V-r_W-V+W}\pi_X^*\sE.
  \]
  Remark~\ref{rem:MultThom} then extends directly to virtual bundles. 

If we have an exact sequence $0\to V'\to V\to V''\to0$, then the corresponding isomorphisms $\Sigma^V\simeq \Sigma^{V'\oplus V''}$ and  $\det V\simeq \det(V'\oplus V'')$ transform $\vartheta_V^\sE$ to $\vartheta_{V'\oplus V''}^\sE$. 
\end{remark}

\begin{remark}
  \label{rem:Oriented}
  If the $\SL$-orientation on $\sE$ extends to a  $\GL$-orientation, then all the   results of this section for $\SL$-oriented theories hold for $\sE$ in simplified form: we can omit all the twisting by line bundles and replace $\Sigma^{1-\det V}\pi_X^*\sE$ with $\pi_X^*\sE$ using the Thom class $\th_{\det V}$ to define an isomorphism $\Sigma_{\T}\pi_X^*\sE\simeq \Sigma^{\det V}\pi_X^*\sE$.
  \end{remark}

We close this section with one last result about twisted $\sE$-cohomology in the $\SL$-oriented setting.

\begin{proposition}
  Let $X \in \Sm_B$, let $L, M$ be two line bundles on $X$, and let $Z\subseteq X$ a closed subset. Then there is a natural isomorphism
  \[
    \psi_{L,M}:\sE^{*,*}_Z(X;L)\to \sE^{*,*}_Z(X;L\otimes M^{\otimes 2}).
  \]
\end{proposition}

\begin{proof}
  Let $s:X\to L\oplus M$ be the zero-section. We have the Thom isomorphisms
  \[
    \sE^{*,*}_Z(X;L)\simeq \sE^{*+4,*+2}_{s(Z)}(L\oplus M;M^{-1}),\quad
    \sE^{*,*}_Z(X;L\otimes M^{\otimes 2})\simeq \sE^{*+4,*+2}_{s(Z)}(L\oplus M;M).
  \]
  Replacing $L \oplus M$ by $X$ and $M$ by $L$, this reduces us to showing that there is a natural isomorphism
  \[
    \psi_L:\sE^{*,*}_Z(X;L)\to \sE^{*,*}_Z(X;L^{-1})
  \]
 
  For this, we follow the proof of \cite[Lemma 2]{AnanSL}. We have the morphism of $X$-schemes
  \[
    L\oplus L^{-1}=L\times_XL^{-1}\xr{\mu} X\times_B\A^1
  \]
  defined by $\mu(x, y)=x\cdot y$; let $Y:=\mu^{-1}(X\times 1)$. Setting $L_0:=L\setminus 0_L$ and $L^{-1}_0:=L^{-1}\setminus 0_L$, we see that $Y$ is a closed subscheme of  $L\times_XL^{-1}$ projecting isomorphically to $L_0$ via $p_1$ and isomorphically to $L_0^{-1}$ via $p_2$.

  Consider the commutative diagram
  \[
    \xymatrix{
      Y\ar[r]\ar[d]_{p_1}& L\times_XL^{-1}\ar[d]_{p_1}\ar[r]& L\times_XL^{-1}/Y\ar[d]^{\bar{p}_1}\\
      L_0\ar[r] & L\ar[r] & L/L_0
    }
  \]
  whose rows are cofiber sequences. As the first two vertical maps are isomorphisms in $\H(B)$, the map $\bar{p}_1$ induces an isomorphism
  \[
    \bar{p}_1/B:  (L\times_BL^{-1}/Y)/B\to \Th(L)/B
  \]
  in $\SH(B)$.  Similarly, we have the isomorphism
  \[
    \bar{p}_2/B:(L\times_BL^{-1}/Y)/B\to  \Th(L^{-1})/B
  \]
  in $\SH(B)$. Replacing $X$ with $U:=X\setminus Z$, we have the isomorphisms
  \[
    \bar{p}_{1U}/B: (L\times_BL^{-1}\times_XU/Y\times_XU)/B\to  \Th(L\times_XU)/B.
  \]
  and 
  \[
    \bar{p}_{2U}/B:  (L\times_BL^{-1}\times_XU/Y\times_XU)/B\to\Th(L^{-1}\times_XU)/B
  \]
  It follows that the arrows in following diagram are isomorphisms in $\SH(B)$ after applying $-/B$:
  \[
    \xymatrix{
      &\Th(L)/\Th(L\times_XU)\\
      \hbox to 150pt{$\left(L\times_XL^{-1}/Y\right)/\left(L\times_XL^{-1}\times_XU/Y\times_XU\right)$\hss}\ar[ur]^{\overline{\bar{p}}_1}\ar[dr]_{\overline{\bar{p}}_2}\\
      &\Th(L^{-1})/\Th(L^{-1}\times_XU)
    }
  \]
  Finally, applying $\Hom_{\SH(B)}(-, \Sigma^{*+2,*+1}\sE)$ gives the desired isomorphism
  \[
    \psi_L:\sE^{*,*}_Z(X;L)\to \sE^{*,*}_Z(X;L^{-1}). \qedhere
  \]
\end{proof}
  
\section{Projective pushforward in twisted cohomology}
\label{sec:Push}

In this section, we describe how one gets projective pushforward maps in twisted $\sE$-cohomology for $\sE$ an $\SL$-oriented motivic spectrum. We rely on the six-functor formalism. This is a bit different from the treatment of projective pushforward given by Ananyevskiy in \cite{AnanSL}: in that treatment, one relies on the factorization of an arbitrary projective morphism $Y\to X$ into a closed immersion $Y\to X\times\P^N$ followed by a projection $X\times\P^N\to X$. This factorization property will however reappear in our treatment when we discuss the uniqueness of the pushforward maps in \S\ref{sec:SLOrCohThy}.

We continue to work over a noetherian separated base scheme $B$ of finite Krull dimension.


\begin{lemma}\label{lem:PushforwardPurity}
Let $s:Y\to X$ be  a section of a smooth morphism $p:X\to Y$ and let $\eta:\id\to s_*s^*$, $\epsilon:s^*s_*\to \id$ be unit and counit of adjunction.  Then the composition
\[
 \id_{\SH(Y)}\xr{\sim} p_!\circ s_!= p_!\circ s_*\xr{s^*}s^*\circ s_*\xr{\epsilon}\id_{\SH(Y)}
\]
is the identity. Here the morphism $s^*:p_!\to s^*$ is the one constructed in Remark~\ref{rmk:BMFunctor}.
\end{lemma}

\begin{proof} The functor $s^*s_*$ is an equivalence \cite[Corollary 4.19]{Hoyois6}. As a general property of adjoint functors, $s_*\epsilon s^*\circ \eta s_*s^*=\id_{s_*s^*}$ (see, e.g., \cite[pg. 134]{MacLane}), hence  $s_*\epsilon s^*s_*\circ \eta s_*s^*s_*=\id_{s_*s^*s_*}$ and thus  $s_*\epsilon\circ \eta s_*=\id_{s_*}$.  The result follows from the commutative diagram
\[
\xymatrix{
&\id\ar[d]^\wr\\
p_{X!}\circ s_*\ar[d]_{p_{X!}\eta s_*}\ar[dr]^{\id}\ar@/_60pt/[ddd]_{s^*}&p_{X!}\circ s_!\ar@{=}[l] \ar@{=}[d]\\
p_{X!}\circ s_*s^*s_*\ar[r]^{p_{X!} s_*\epsilon}\ar@{=}[d] &p_{X!}\circ s_*\ar@{=}[d]\\
p_{X!}\circ s_!\circ s^*\circ s_*\ar[d]_\wr\ar[r]^{p_{X!} s_!\epsilon}&p_{X!}\circ s_!\ar[d]^\wr\\
s^*\circ s_*\ar[r]_\epsilon&\id
}
\]
\end{proof}

\begin{remark}\label{rem:Pushforward} Let $s:Y\to V$ be the zero-section of a vector bundle $p:V\to Y$. Then the canonical isomorphism $\id_{SH(Y)}\simeq p_!s_*$ is equal to the composition
\[
\id_{\SH(Y)}\simeq \Sigma^{-V}\Sigma^V=\Sigma^{-V}p_\#s_*\simeq p_\#\Sigma^{-p^*V}s_*\simeq
p_!s_*
\]
Indeed $\Sigma^{-V}\Sigma^V\to \id_{\SH(Y)}$ is the counit of the adjunction $\Sigma^{-V}\dashv\Sigma^V$, corresponding to $\id:\Sigma^V\to \Sigma^V$,  $\Sigma^{-V}p_\#s_*\to  \id_{\SH(Y)}$ is the counit of the  adjunction $\Sigma^{-V}=s^!p^*\dashv p_\#s_*$. The functors $\Sigma^{-V}$ and $p_\#$ are both left adjoints, so $\Sigma^{-V}p_\#$ is left adjoint to $s_*$ and the counit of the adjunction 
\[
\Sigma^{-V}p_\#\dashv p^*\Sigma^V\simeq \Sigma^{p^*V}p^*\simeq p^!\simeq s_*
\]
is the same as that of $\Sigma^{-V}\dashv p_\#s_*$. Composing with the isomorphism $p_!\simeq p_\#\Sigma^{-p^*V}\simeq \Sigma^{-V}p_\#$, we see that the counit of the adjunction $p_!\dashv s_*$ is induced from that of 
$\Sigma^{-V}p_\#\dashv s_*$, and thus agrees with the canonical isomorphism $p_!s_*\simeq \id_{\SH(Y)}$.
\end{remark}

\begin{lemma}
  \label{lem:Duality}
  Let $(\sE, \th_{(-)})$ be an $\SL$-oriented motivic spectrum in $\SH(B)$. Suppose given $X\in \Sm_B$ of dimension $d_X$ over $B$, $i:Z\to X$ a closed subset, and $p:L\to X$ a line bundle. Then the isomorphism
  \[
    \vartheta_{L-T_{X/B}}^\sE:\Sigma^{1-\det(L-T_{X/B})}\pi_X^*\sE\to 
    \Sigma^{r_{L-T_{X/B}}-L+T_{X/B}}\pi_X^*\sE
  \]
  induces an isomorphism
  \[
    \rho_{X, Z, L}:\sE^{a,b}_Z(X; \omega_{X/B}\otimes L)\simeq \sE^{a-2d_X,b-d_X}(X_Z/B_\BM; L),
  \]
  where the right-hand side is as defined in Remark~\ref{rmk:TwistGeneral}.
\end{lemma}

\begin{proof}
  We have $\det(L-T_{X/B})=\det^{-1}(T_{X/B})\otimes L=\omega_{X/B}\otimes L$ and $r_{L-T_{X/B}}=1-d_X$. Moreover, we have canonical isomorphisms
  \[
    \sE^{a,b}_Z(X; \omega_{X/B}\otimes L)\simeq \Hom_{\SH(X)}(i_*(1_Z), S^{a,b}\wedge\Sigma^{1-\omega_{X/B}\otimes L}\pi_X^*\sE)
  \]
  and
  \[
    \sE^{a-2d_X,b-d_X}(X_Z/B_\BM; L)\simeq
    \Hom_{\SH(X)}(i_*(1_Z), S^{a,b}\wedge\Sigma^{(1-d_X)+T_{X/B}-L}\pi_X^*\sE).
  \]
  Finally,  $\vartheta_{L-T_{X/B}}^\sE$ induces the isomorphism
  \[
    S^{a,b}\wedge \vartheta_{L-T_{X/B}}^\sE:S^{a,b}\wedge\Sigma^{1-\omega_{X/B}\otimes L}\pi_X^*\sE\to S^{a,b}\wedge\Sigma^{(1-d_X)+T_{X/B}-L}\pi_X^*\sE
  \]
  which completes the proof.
\end{proof}

Using the isomorphisms $\rho_{X, Z,L}$, we make the following definition.

\begin{definition}
  Let $(\sE, \th_{(-)})$ be an $\SL$-oriented motivic spectrum in $\SH(B)$, let $f:Y\to X$ be a proper morphism of relative dimension $d$ in $\Sm_B$, let $L\to X$ be a line bundle, and let $Z\subset X$ be a closed subset. Define 
  \[
    f_*:\sE^{a,b}_{f^{-1}(Z)}(Y, \omega_{Y/B}\otimes f^*L)\to 
    \sE^{a-2d,b-d}_Z(X, \omega_{X/B}\otimes L)
  \]
  to be the unique map making the diagram
  \[
    \xymatrix{
      \sE^{a-2d_Y,b-d_Y}(Y_{f^{-1}(Z)}/B_\BM; f^*L)\ar[r]^-{(f^*)^*}&\sE^{a-2d_Y,b-d_Y}(X_Z/B_\BM; L)\\
      \sE^{a,b}_{f^{-1}(Z)}(Y, \omega_{Y/B}\otimes f^*L)\ar[u]^{\rho_{Y, f^{-1}Z, f^*L}}\ar[r]_{f_*}&
      \sE^{a-2d,b-d}_Z(X, \omega_{X/B}\otimes L)\ar[u]_{\rho_{X, Z, L}}
    }
  \]
  commute.
\end{definition}

Let $p:V\to Y$ be a  rank $r$ vector bundle on some $Y\in \Sm_B$, with 0-section $s:Y\to V$. Letting $L=\det V$, the exact sequence
\[
  0\to p^*V\to T_{V/B}\xr{dp}p^*T_{Y/B}\to 0
\]
gives the canonical isomorphism $\omega_{V/B}\simeq p^*(\det^{-1} V\otimes \omega_{Y/B})$, or 
$p^*\det^{-1} V\simeq \omega_{V/B}\otimes\omega^{-1}_{Y/B}$. Letting $(\sE, \th)$ be an $\SL$-oriented motivic spectrum, we have the pushforward map
\[
  s_*:\sE^{a,b}(Y)\to \sE^{ a+2r, b+r}(V, p^*\det^{-1} V),
\]
and the version with supports,
\[
  s_*:\sE^{a,b}(Y)=\sE^{a,b}_Y(Y)\to \sE^{ a+2r, b+r}_{0_V}(V, p^*\det^{-1} V).
\]

\begin{lemma}
  \label{lem:InclThomIso}
  Let $1_Y^\sE\in \sE^{0,0}(Y)$ be the unit $\pi_{Y/B}^*(u)$. Then 
  \[
    s_*(1_Y^\sE)=\th_V \in \sE^{2r, r}_{0_V}(V, p^*\det^{-1} V). 
  \]
  As consequence, $s_*(1_Y^\sE)$ in $\sE^{2r, r}(V, p^*\det^{-1} V)$ is the image of $\th_V$
  under the ``forget supports'' map 
  \[
    \sE^{2r, r}_{0_V}(V, p^*\det^{-1} V)\to  \sE^{2r, r}(V, p^*\det^{-1} V).
  \]
\end{lemma}

\begin{proof}
  The exact sequence
  \[
    0\to p^*V\to T_{V/B}\xr{dp}p^*T_{Y/B}\to 0
  \]
  gives us the isomorphism  
  \[
    \omega_{V/B}^{-1}\otimes\det^{-1} V  \simeq  p^*\omega_{Y/B}^{-1}.
  \]
  Keeping this in mind,  we have the following commutative diagram defining $s_*$:
  \begin{equation}\label{eqn:Compat1}
    \xymatrix{
      \sE^{-2d_Y,-d_Y}(Y/B_\BM; \omega_{Y/B}^{-1})\ar[r]^-{(s^*)^*}&\sE^{-2d_Y,-d_Y}_{0_V}(V/B_\BM,  p^*\omega_{Y/B}^{-1})\\
      \sE^{0,0}_Y(Y)\ar[u]^{\rho_{Y,Y, s^*p^*\omega_{Y/B}^{-1}}}\ar[r]_{s_*}\ar[rd]_{s_*}&
      \sE^{2r,r}_{0_V}(V, \det^{-1} V)\ar[u]_{\rho_{V, 0_V, p^*\omega_{Y/B}^{-1}}}\ar[d]^{\vbox{\tiny \noindent forget\\supports}}\\
      &\sE^{2r,r}(V,\det^{-1} V )
    }
  \end{equation}
  Here the lower $s_*$ is the one we are considering and the upper $s_*$ is the map with supports. Thus, we need to show that the upper $s_*$ satisfies $s_*(1_Y^\sE)=\th_V$.
  
  We will be using the isomorphisms
  \begin{align}\label{align:Isos}
    \\\notag
 &\sE^{-2d_Y,-d_Y}(Y/B_\BM; \omega_{Y/B}^{-1})\simeq \Hom_{\SH(Y)}(1_Y, \Sigma^{1-d_Y+T_{Y/B}-\omega_{Y/B}^{-1}}\pi_Y^*\sE),\\\notag
 &\sE^{-2d_Y,-d_Y}_{0_V}(V/B_\BM,  p^*\omega_{Y/B}^{-1})\simeq
   \Hom_{\SH(V)}(s_*(1_Y), \Sigma^{1-d_Y+T_{V/B}- p^*\omega_{Y/B}^{-1}}\pi_V^*\sE),\\\notag
 &\sE^{2r,r}_{0_V}(V,\det^{-1} V )\simeq \Hom_{\SH(V)}(s_*(1_V), \Sigma^{r+1-\det^{-1}V}\pi_V^*\sE).
  \end{align}

By Lemma~\ref{lem:PushforwardPurity}, the composition
 \[
 \pi_{Y!}\xr{\sim}\pi_{V!}\circ s_!\xr{\pi_{V!}\alpha}\pi_{V!}\circ s_*\xr{s^*}\pi_{Y!}
 \]
is the identity. Evaluating at $1_Y$ gives the commutative diagram
\begin{equation}\label{eqn:Compat2} 
\xymatrix{
Y_\BM\ar[r]^-{\phi}_--\sim\ar@{=}[d]&\pi_{V!}(s_*(1_Y))\ar[dl]^{s^*}\ar@{=}[r]&V_{0_V}/B_{\BM}\\
Y_\BM}
\end{equation}
and the isomorphisms $\phi$ induces the isomorphism
\[
\phi^*:\sE^{-2d_Y, -d_Y}_{0_V}(V/B_\BM, p^*\omega_{Y/B}^{-1})\to
\sE^{-2d_Y, -d_Y}(Y_\BM, \omega_{Y/B}^{-1})
\]

The isomorphism $\rho_{V, 0_V, p^*\omega_{Y/B}^{-1}}$ is the map induced on $\Hom_{\SH(V)}(s_*(1_Y),-)$ by the isomorphism
\[
\vartheta_{p^*\omega^{-1}_{Y/B}- T_{V/B}}:\Sigma^{1-\det^{-1} V}\pi_V^*\sE\to \Sigma^{1-d_V+T_{V/B}-p^*\omega^{-1}_Y/B}\pi_V^*\sE.
\]
We have the isomorphisms
\[
\Sigma^{r-V}\vartheta_{-V}: \Sigma^{r+1-V-\det^{-1}V}\pi_Y^*\sE\to \pi_Y^*\sE
\]
\[
\Sigma^{r-V}\vartheta_{\omega^{-1}_{Y/B}-T_{Y/B}-V}:
\Sigma^{r+1-V-\det^{-1}V}\pi_Y^*\sE
 \to \Sigma^{r+1-d_Y+T_{Y/B}-\omega_{Y/B}}\pi_Y^*\sE
\]
and
\[
\vartheta_{\omega_{Y/B}-T_{T/B}}:\pi_Y^*\sE\to \Sigma^{1-d_Y+T_{Y/B}-\omega_{Y/B}}\pi_Y^*\sE.
\]
The first one induces an isomorphism
\[
\rho_{-V}:\Hom_{\SH(Y)}(1_Y, \Sigma^{r+1-V-\det^{-1}V}\pi_Y^*\sE)\to \sE^{0,0}(Y)\
\]
the second   an isomorphism
\[
\rho_{T_{Y/B}+V-\omega_{Y/B}^{-1}}:\Hom_{\SH(Y)}(1_Y, \Sigma^{r+1-V-\det^{-1}V}\pi_Y^*\sE)\to
\sE^{-2d_Y, -d_Y}(Y/B_\BM, \omega^{-1}_{Y/B})
\]
while the third induces the isomorphism 
\[
\rho_{Y,Y, \omega_{Y/B}^{-1}}:\sE^{0,0}(Y)\to \sE^{-2d_TY, -d_Y}(Y/B_\BM, \omega^{-1}_{Y/B})
\]

Altogether these maps and isomorphisms gives the diagram of isomorphisms
\begin{equation}\label{eqn:CommDiagr}
\xymatrixcolsep{5pt}
\xymatrix{
\sE^{-2d_Y, -d_Y}_{0_V}(V/B_\BM, p^*\omega_{Y/B}^{-1})\ar[r]^{\phi^*}& \sE^{-2d_Y, -d_Y}(Y/B_\BM, \omega^{-1}_{Y/B})\\
&&\sE^{0,0}(Y)\ar[ul]_-{\rho_{Y,Y,\omega_{Y/B}^{-1}}}\\
\sE_{0_V}^{2r,r}(V, \det^{-1}V)\ar[uu]^{\rho_{V, 0_V, \det^{-1}V}}\ar[r]_-\psi&
\Hom_{\SH(Y)}(1_Y, \Sigma^{r+1-V-\det^{-1}V}\pi_Y^*\sE)\ar[uu]^{\rho_{T_{Y/B}+V-\omega_{Y/B}}}
\ar[ur]_-{\rho_{-V}}}
\end{equation}
The  triangle commutes by the functoriality of $\vartheta_{-}$, as expressed by Remark~\ref{rem:MultThom}. 

To see that   square commutes, we have the diagram
\[
\xymatrixcolsep{10pt}
\xymatrix{
\Sigma^{1-d+T_V-p^*\omega_{Y/B}^{-1}}\pi_V^*\sE\ar[r]^-\sim&
\Sigma^{1-d+p^*(T_Y+V-\omega_{Y/B}^{-1})}\pi_V^*\sE\ar[r]^-\sim&
p^*\Sigma^{1-d+T_Y+V-\omega_{Y/B}^{-1}}\pi_Y^*\sE\\
\Sigma^{r+1-p^*\det^{-1}V}\pi_V^*\sE\ar@{=}[r]\ar[u]_{\Sigma^r_{\T}\vartheta_{p^*\omega_{Y/B}^{-1}-T_{V/B}}}&
\Sigma^{r+1-p^*\det^{-1}V}\pi_V^*\sE\ar[r]_\sim\ar[u]_{\Sigma^r_{\T}\vartheta_{p^*(\omega_{Y/B}^{-1}-T_{Y/B}-V)}}&p^*\Sigma^{r+1-\det^{-1}V}\pi_Y^*\sE\ar[u]_{p^*\Sigma^r_{\T}\vartheta_{\omega_{Y/B}^{-1}-T_{Y/B}-V}}
}
\]
The first square commutes using the exact sequence $0\to p^*V\to T_{V/B}\to p^*T_{Y/B}\to0$ and the second by the naturality of $\vartheta_{-}$.   Applying the adjunction $p_\#\dashv p^*$, the identity $p_\#s^*=\Sigma^V$ and applying $\Sigma^{-V}$ to yield the isomorphism $[\Sigma^Vx, y]_{\SH(Y)}\simeq [x,\Sigma^Vy]_{\SH(Y)}$, applying $\Hom_{\SH(V)}(s_*(1_Y), -)$ to the last map gives the commutative square
\[
\xymatrix{
\Hom_{\SH(V)}(s_*(1_Y), p^*\Sigma^{1-d+T_Y+V-\omega_{Y/B}^{-1}}\pi_Y^*\sE)\ar[dr]^\sim\\
&\kern-50pt\Hom_{\SH(Y)}(1_Y, \Sigma^{1-d+T_Y-\omega_{Y/B}^{-1}}\pi_Y^*\sE)\\
\Hom_{\SH(V)}(s_*(1_Y), p^*\Sigma^{r+1-\det^{-1}V}\pi_Y^*\sE)\ar[uu]^{p^*\Sigma^r_{\P^1}\vartheta_{\omega_{Y/B}^{-1}-T_{Y/B}-V*}}\ar[rd]_\sim\\
&\kern-50pt
\Hom_{\SH(Y)}(1_Y,\Sigma^{r+1-V-\det^{-1}V}\pi_Y^*\sE)
\ar[uu]_{\Sigma^{r-V}\vartheta_{\omega_{Y/B}^{-1}-T_{Y/B}-V*}}
}
\]
Applying $\Hom_{\SH(V)}(s_*(1_Y), -)$ to the first diagram, putting these two diagrams together and using the isomorphisms \eqref{align:Isos} and Remark~\ref{rem:Pushforward}  gives the commutativity of the square in \eqref{eqn:CommDiagr}.

It follows from the commutativity of  \eqref{eqn:Compat2} that $\phi^*\circ(s^*)^*\circ \rho_{Y,Y,\omega_{Y/B}^{-1}}=\rho_{Y,Y,\omega_{Y/B}^{-1}}$. The commutativity of \eqref{eqn:Compat1} and \eqref{eqn:CommDiagr} then shows that $\rho_{-V}\circ\psi\circ s_*=\id$.  By Remark~\ref{rem:CanonicalThom}, and Remark~\ref{rem:MultThom} the map $\rho_{-V}\circ\psi$ is the inverse of the canonical Thom isomorphism $\vartheta_V:\sE^{0,0}(Y)\to \sE^{2r, r}_{0_V}(V, \det^{-1}V)$. Thus $s_*=\vartheta_V$ so $s_*(1^\sE_Y)=\th_V$.\qedhere \end{proof}

\begin{remark}
  \label{rem:Oriented2}
  If we have a $\GL$-orientation on $\sE$, we have functorial pushforward maps
  \[
    f_*:\sE^{a,b}_W(Y)\to \sE^{a-2d, b-d}_Z(X)
  \]
  for $f:Y\to X$ a projective morphism in $\Sm_B$, of relative dimension $d$, with $W\subset Y$, $Z\subset X$ closed subsets with $f(W)\subset Z$. All the results of this section hold in the oriented context after deleting the twist by line bundles. This follows from Remark~\ref{rem:Oriented}.
\end{remark}

\section{Motivic Gau{\ss}-Bonnet}
\label{sec:GaussBonnet}

\begin{definition}
  Let $p:V\to X$ be a rank $r$ vector bundle on some $X\in \Sm_B$, and let $\sE\in \SH(B)$ be an $\SL$-oriented motivic ring spectrum. The {\em Euler class} $e^\sE(V)\in \sE^{2r,r}(X,\det^{-1}V)$ is defined as
  \[
    e^\sE(V):=s^*s_*(1^\sE_X);\quad 1^\sE_X\in \sE^{0,0}(X)\text{ the unit.}
  \]
\end{definition}

\begin{remark}\label{rem:EulerThom} 
By Lemma~\ref{lem:InclThomIso}, $e^\sE(V):=s^*s_*(1^\sE_X)=\bar{s}^*\th_V$, where $\bar{s}:X\to \Th_X(V)$ is the map induced by $s$.
\end{remark}

\begin{theorem}[Motivic Gau{\ss}-Bonnet]\label{thm:GaussBonnet}
  Let $\sE\in\SH(B)$ be an $\SL$-oriented motivic ring spectrum, $\pi_X:X\to B$ a  smooth and projective $B$-scheme, let $u_\sE:1_B\to \sE$ be the unit map.  Then 
  \[
    \pi_{X/B*}(e^\sE(T_{X/B}))=u_{\sE*}(\chi(X/B))\in \sE^{0,0}(B).
  \]
\end{theorem}

\begin{proof}

  We have the canonical Thom isomorphism
  \[
    \vartheta^\sE_{-T_{X/B}}: \sE^{a,b}(X;\omega_{X/B})\to \sE^{a-2\dim X, b-\dim X}(\Th(-T_{X/B})).
  \]
  By Lemma~\ref{lem:simp}, it suffices to show that the map
  \[
    \beta_{X/B}^*:\sE^{0,0}(X)\to \sE^{0,0}(\Th(-T_{X/B}))
  \]
  sends $1^\sE_X$ to $\vartheta^\sE_{-T_{X/B}}(e^\sE(T_{X/B}))$; by Remark~\ref{rem:EulerThom}, this is the same as $\vartheta_{-T_{X/B}}(\bar{s}^*\th_{T_{X/B}})$, where $\bar{s}:X_+\to \Th(T_{X/B})$ is the map induced by the zero-section  $s:X\to T_{X/B}$.

We use our description  of $\beta_{X/B}$ as $\pi_{X\#}$ applied to the composition \eqref{const:betaComp2}. Applying $\Hom_{\SH(X)}(-,\pi_X^*\sE)$ to $\beta_{X/B}$ and using the adjunction $\Hom_{\SH(B)}(\pi_{X\#}(-), \sE)\simeq  \Hom_{\SH(X)}(-,\pi_X^*\sE)$, $ \beta_{X/B}^*$ is given by the composition
\begin{multline*}
\sE^{0,0}(X)\xymatrix{\ar[r]^a_\sim&} \Hom_{\SH(X)}(1_X, \pi_X^*\sE)\xymatrix{\ar[r]^b_\sim&}
\Hom_{\SH(X)}(\Sigma^{-T_{X/B}}\circ\Sigma^{T_{X/B}}(1_X), \pi_X^*\sE)
\\\xymatrix{\ar[r]^c_\sim&}
\Hom_{\SH(X)}(\Th_X(T_{X/B})), \Sigma^{T_{X/B}}\pi_X^*\sE)
\\
\xr{\bar{s}^*}
\Hom_{\SH(X)}(1_X, \Sigma^{T_{X/B}}\pi_X^*\sE)
\simeq
\Hom_{\SH(X)}(\Sigma^{-T_{X/B}}(\Th_X(T_{X/B})),\pi_X^*\sE)
\end{multline*}
where the isomorphisms $a,b, c$ are the canonical ones.

The functoriality of the canonical Thom isomorphisms gives us the commutative diagram 
\[
\xymatrix{
\sE^{0,0}(X)\ar[r]^-{\vartheta_{T_{X/B}}}\ar[d]^a_\wr&\sE^{2d_X, d_X}(\Th(T_{X/B}), \omega_{X/B})\ar[dd]^{\vartheta^\sE_{-T_{X/B}}}\\
\Hom_{\SH(X)}(1_X, \pi_X^*\sE)\ar[d]^b_\wr&\\
\Hom_{\SH(X)}(\Sigma^{-T_{X/B}}\Sigma^{T_{X/B}}(1_X),\pi_X^*\sE)\ar[r]_-c^-\sim&
\Hom_{\SH(X)}(\Th_X(T_{X/B}),\Sigma^{T_{X/B}}\pi_X^*\sE)
}
\]
  Thus 
\[
(c\circ b\circ a)(1^\sE_X)= \vartheta^\sE_{-T_{X/B}}(\th_{T_{X/B}}).
\]

Applying $\vartheta^\sE_{-T_{X/B}}$ as above gives us the commutative diagram
\[
\xymatrixcolsep{40pt}
\xymatrix{
\sE^{2d_X, d_X}(\Th(T_{X/B}), \omega_{X/B}) \ar[r]^-{\bar{s}^*}
\ar[d]_{\vartheta^\sE_{-T_{X/B}}}
&\sE^{2d_X, d_X}(X, \omega_{X/B})\ar[d]^{\vartheta^\sE_{-T_{X/B}}}\\
\Hom_{\SH(X)}(\Th_X(T_{X/B}), \Sigma^{T_{X/B}}\pi_X^*\sE)
\ar[r]_-{\bar{s}^*}\ar[d]_\wr
&\Hom_{\SH(X)}(1_X, \Sigma^{T_{X/B}}\pi_X^*\sE)\ar[d]^\wr\\
\Hom_{\SH(X)}(\Sigma^{-T_{X/B}}\Th_X(T_{X/B}), \pi_X^*\sE)
\ar[r]_-{\Sigma^{-T_{X/B}}(\bar{s}^*)} 
&\Hom_{\SH(X)}(\Sigma^{-T_{X/B}}(1_X), \pi_X^*\sE)
}
\]
and thus  
\[
\beta_{X/B}^*(1^\sE_X)=\vartheta_{-T_{X/B}}(\bar{s}^*(\th_{T_{X/B}}))=\vartheta_{-T_{X/B}}(e^\sE(T_{X/B})), 
\]
as desired.
\end{proof}

\section{$\SL$-oriented cohomology theories}
\label{sec:SLOrCohThy}

Our ultimate goal is to apply the Gau{\ss}-Bonnet theorem of \S\ref{sec:GaussBonnet} when projective pushforwards are defined on a representable cohomology theory in some concrete manner, not necessarily  relying on the six-functor formalism. For this, we need a suitable axiomatization for such theories; we use a modification of  the axioms of Panin-Smirnov \cite{PaninII, PaninI}. As before, our base-scheme $B$ is a noetherian, separated scheme of finite Krull dimension.

\begin{definition}
  We let $\SmL_B$ denote the category of triples $(X, Z, L)$ with $X$ in $\Sm_B$, $Z\subset X$ a closed subset and $L\to X$ a line bundle. A morphism $(f,\tilde{f}):(X, Z, L)\to (Y, W, M)$ is a morphism $f:X\to Y$ with $Z\supset f^{-1}(W)$, together with an  isomorphism of line bundles $\tilde{f}:L\to f^*M$.  We let $\SmL^\pr_B$ denote the category with the same objects as   $\SmL_B$, but with morphisms $(f; \tilde{f}):(X,Z, L)\to (Y, W, M)$ a proper morphism $f:X\to Y$ in $\Sm_B$, with $f(Z)\subset W$, and $\tilde{f}:L\to f^*M$ an isomorphism of line bundles.
\end{definition}

\begin{definition}
  An {\em $\SL$-oriented cohomology theory} on $\Sm_B$ consists of the following data:
  \begin{enumerate}[label=(D\arabic*)]
  \item A functor $H^{*,*}:\SmL_B^\op\to \BiGr\Ab$, $(X, Z, L)\mapsto H^{*,*}_Z(X;L)$; we often write $f^*$ for $H^{*,*}(f,\tilde{f})$.
  \item A functor $H_{*,*}:\SmL^\pr_B\to \Gr\Ab$, $(X, Z, L)\mapsto H_{*,*}^Z(X,L)$; we often write $f_*$ for $H_{*,*}(f,\tilde{f})$.
  \item Natural isomorphisms, for $X$ of dimension $d_X$
    \[
      H^{2d_X-n, d_X-m}_Z(X,\omega_{X/B}\otimes L)\xr{\alpha_{X,Z,L}} H_{n,m}^Z(X, L).
    \]
  \item An element $1\in H^{0,0}_B(B;\sO_B)$. For $x:=(X, Z, L), y:=(Y, W, M)$ in  $\SmL_B$, a bigraded cup product map
    \[
      \cup_{x,y}:H^{*,*}_Z(X,L)\otimes H^{*,*}_W(Y,M)\to H^{*,*}_{Z\times W}(X\times_BY, p_1^*L\otimes p_2^*M)
    \]
  \item For $Z\subset W$ closed subsets of an $X\in \Sm_B$ a bigraded boundary map
    \[
      \delta_{X, W, Z}^{*,*}:H^{*, *}_{Z\setminus W}(X\setminus W;j_W^*L)\to H^{*+1,*}_W(X, L)
    \]
    We write $H^{*,*}(X,L)$ for $H^{*,*}_X(X,L)$ and $H^{*,*}_Z(X)$ for   $H^{*,*}_Z(X,\sO_X)$; we use the analogous notation for $H_{*,*}$. We write $\cup$ for $\cup_{x,y}$ and $\delta$ for $\delta_{X,Z,L}$ when the context makes the meaning clear.
  \end{enumerate}
  
  For $f:Y\to X$ a proper map of relative dimension $d$ in $\Sm_B$, with $Z\subset X$, $W\subset Y$  closed subsets with $f(W)\subset Z$ and $L\to X$ a line bundle, combining D2 and D3 gives us pushforward maps
  \[
    f_*:H^{*, *}_W(Y, \omega_{Y/B}\otimes f^*L)\to H^{*-2d, *-d}_Z(X, \omega_{X/B}\otimes L)
  \]
  defined as the composition
  \begin{multline*}
    H^{*, *}_W(Y, \omega_{Y/B}\otimes f^*L)\xr{\alpha_{Y,W,f^*L}^{-1}}
    H_{2d_Y-*, d_Y-*}^W(Y,   f^*L)\\\xr{f_*}H_{2d_Y-*, d_Y-*}^Z(X,   L)
    \xr{\alpha_{X,Z,L}}H^{*-2d, *-d}_Z(X, \omega_{X/B}\otimes L).
  \end{multline*}

  These data are required to satisfy the following axioms:
  \begin{enumerate}[label=(A\arabic*)]
  \item $H^{*,*}$ and $H_{*,*}$ are additive: $H^{*,*}$ transforms disjoint unions to products and $H_{*,*}$ transforms disjoint unions to coproducts; in particular, $H^{*,*}_Z(\0,L)=0$ and  $H_{*,*}^Z(\0, L)=0$.
  \item Let 
    \[
      \xymatrix{
        Y'\ar[d]_{f'}\ar[r]^-{g'}&Y\ar[d]^f\\
        X'\ar[r]_-g&X
      }
    \]
    be a cartesian diagram in $\Sch_B$, with $X, Y, X', Y'$ in $\Sm_B$ (sometimes called a {\em transverse} cartesian diagram in $\Sm_B$) and with $f, f'$ proper of relative dimension $d$. This gives us the isomorphism
    \[
      f^{\prime*}\omega_{X'/X}\simeq \omega_{Y'/Y}.
    \]
    Let $Z\subset X$ be a closed subset, let $W\subset Y$ be a closed subset with $f(W)\subset Z$, let $Z'=g^{-1}(Z)$, $W'=g^{\prime-1}(W)$. Let $L\to X$ be a line bundle on $X$ and let $L'=g^{\prime*}(L)$. Then the diagram
    \[
      \xymatrix{
        H^{*, *}_{W'}(Y', \omega_{Y'/B}\otimes\omega_{Y'/Y}^{-1}\otimes g^{\prime*} L')\ar[d]^{  f'_*}&\ar[l]_-{g^{\prime*}}H^{*, *}_W(Y, \omega_{Y/B}\otimes f^*L)\ar[d]^{ f_*}\\
        H^{*-2d, *-d}_{Z'}(X', \omega_{X'/B}\otimes\omega_{X'/X}^{-1}\otimes  L')&\ar[l]^-{g^*}H^{*-2d, *-d}_Z(X, \omega_{X/B}\otimes L)
      }
    \]
    commutes.
  \item For $Z\subset W$ closed subsets of an $X\in \Sm_B$,  let $U=X\setminus Z$ with inclusion $j:U\to X$.  For $L\to X$ a line bundle, this gives us the morphisms $(\id, \id):(X, W, L)\to (X, Z, L)$ and $(j, \id):(U, W\setminus Z, j^*L)\to (X, W, L)$. Then the sequence
    \begin{multline*}
      \ldots\xr{\delta_{Z, W, X}}H^{*, *}_Z(X, L)\to H^{*,*}_W(X, L)\\
      \xr{j^*} H^{*,*}_{W\setminus Z}(U, j^*L)\xr{\delta_{Z, W, X}}H^{*+1, *}_Z(X, L)\to\ldots
    \end{multline*}
    is exact.  Moreover,  the maps $\delta_{Z, W, X}$ are natural with respect to the pullback maps $g^*$ and the proper pushforward maps $f_*$.
  \item Let $i:Y\to X$ be a closed immersion in $\Sm_B$, let $W\subset Y$ be a closed subset, $L\to X$ a line bundle. Let $Z=i(W)$, giving the morphism $(i, \id):(Y, W, i^*L)\to (X, Z, L)$ in  $\SmL^\pr_B$. Then
    \[
      i_*:H_{*,*}^W(Y, i^*L)\to H_{*,*}^Z(X, L)
    \]
    is an isomorphism.
  \item The cup products $\cup$ of D4 are associative with unit 1. The maps $f^*$ and $f_*$ are compatible with cup products: $(f\times g)^*(\alpha\cup_{x,y}\beta)=f^*(\alpha)\cup_{x,y} g^*(\beta)$. Moreover, using the isomorphisms  of D3, the cup products induce products $\cup^{x,y}$ on $H_{*,*}$  and one has $(f\times g)_*(\alpha\cup^{x,y}\beta)=f_*(\alpha)\cup^{x,y} g_*(\beta)$. Finally, the boundary maps $\delta_{Z, W, X}$ are module morphism: retaining the notation of D4,  for $\alpha\in H^{*, *}_{Z\setminus W}(X\setminus W;j_W^*L)$ and $\beta\in H^{*, *}_{T}(Y, M)$, we have
    \[
      \delta_{X\times Y, Z\times T, W\times T}(\alpha\cup  \beta)=
      \delta_{X, Z, W}(\alpha)\cup \beta.
    \]
  \item Let $i:Y\to X$ be a closed immersion in $\Sm_B$ of codimension $c$, $\pi_Y:Y\to B$ the structure map.  Let $1^H_Y\in H^{0,0}(Y)$ be the element $\pi_Y^*(1)$. Then $\vartheta(i):=\alpha_{X, Y}(i_*(1^H_Y))\in H^{2c, c}_Y(X, \det^{-1} N_i)$ is {\em central}, that is, for each $(U, T, M)\in \SmL_B$, and each $\beta\in H^{*,*}_T(U, M)$, we have 
    \[
      \tau^*(\beta\cup\vartheta(i))=\vartheta(i)\cup\beta
    \]
    where $\tau:X\times_BU\to U\times_BX$ is the symmetry isomorphism.
  \item Let $(f,\id):(Y, W, f^*L)\to (X, Z, L)$ be a morphism in $\SmL_B$. Suppose that the induced map $f:Y_W/B\to X_Z/B$ is an isomorphism in $\SH(B)$. Then 
    \[
      f^*:H^{*,*}_Z(X, L)\to H^{*,*}_W(Y, f^*L)
    \]
    is an isomorphism.
  \end{enumerate}
\end{definition}

\begin{remark} It may seem strange that the proper pushforward maps respect products in the sense of (A5); one might rather expect a projection formula. However, (A5) asks that the proper pushforward maps respect {\em external} products, not cup products, and in fact, having the 
pushforward and pullback maps respect products as in (A5) implies the projection formula, as one sees by considering the commutative pentagon associated to a proper morphism $f:Y\to X$ in $\Sm_B$ of relative dimension $d$:
\[
\xymatrix{
Y\ar[dd]_f\ar@/^40pt/[drr]^{\gamma_f:=(f\times\id_X)\circ\Delta_Y}\ar[r]_-{\Delta_Y}&Y\times_BY\ar[dr]_{f\times \id_Y}\\
&&X\times_BY\ar[dl]^{\id_X\times f}\\
X\ar[r]_-{\Delta_X}&X\times_BX
}
\]
Note that the square
\[
\xymatrix{
Y\ar[r]^-{\gamma_f}\ar[d]_f&X\times_BY\ar[d]^{\id_X\times f}\\
X\ar[r]_-{\Delta_X}&X\times_BX
}
\]
 is transverse cartesian. 
If we have closed subsets $Z\subset X$, $W\subset Y$ with $f(W)\subset Z$, and line bundle $L\to X$, the pentagon diagram induces the diagram in cohomology
\[
\xymatrixcolsep{15pt}
\xymatrix{
H^{*,*}_W(Y,\omega_{Y/B}\otimes f^*L)Y\ar[dd]_{f_*}&\ar[l]^-{\Delta_Y^*}
H^{*,*}_{W\times W}(Y\times_BY, \omega_{Y/B}\otimes f^*L)\\\
&&\kern-43pt H^{*,*}_{Z\times W}(X\times_BY, \omega_{Y/B}\boxtimes L)\ar@/_55pt/[ull]_{\gamma_f^*}
\ar[ul]_{\ \ (f\times \id_Y)^*}\ar[dl]^{\ \ (\id_X\times f)_*}\\
H^{*-2d, *-d}_Z(X, \omega_{X/B}\otimes L)&H^{*-2d, *-d}_{Z\times Z}(X\times_BX, \omega_{X/B}\otimes L)\ar[l]^-{\Delta^*_X}
}
\]
Take $\alpha\in H^{a,b}_Z(X, M)$, $\beta\in H^{c,d}_W(Y, \omega_{Y/B}\otimes f^*(L\otimes M^{-1}))$. By functoriality of $(-)^*$ and (A5) for $(-)^*$ we have $\gamma_f^*(\alpha\cup_{X,Y} \beta)=f^*(\alpha)\cup_Y\beta$ and by (A2) and (A5) for $(-)_*$ we have
\[
f_*(f^*(\alpha)\cup_Y\beta)=\Delta_X^*(\id_X\times f)_*(\alpha\cup_{X,Y}\beta)
=\alpha\cup_X f_*(\beta).
\]
Similarly, in the presence of (A2) and (A5) for $(-)^*$, functoriality for $(-)^*$ and $(-)_*$ and the projection formula implies (A5) for $(-)_*$.
\end{remark}

\begin{definition}
  A  {\em  twisted cohomology theory} on $\Sm_B$ is given by the data D1,  D4, D5 above, satisfying the parts of the axioms A1, A3-A7 that only involve $H^{*,*}$. Given an $\SL$-oriented cohomology theory $(H^{*,*}, H_{*,*}, \ldots)$ on $\Sm_B$, one has the underlying twisted cohomology theory $(H^{*,*}, \ldots)$ by forgetting the proper pushforward maps.
\end{definition}

\begin{example}
  The primary example of an $\SL$-oriented cohomology theory on $\Sm_B$ is the one induced by an $\SL$-oriented motivic spectrum $\sE\in \SH(B)$:
  \[
    (X,Z,L)\mapsto \sE^{*,*}_Z(X;L).
  \]
  One defines, for $X\in \Sm_B$ of dimension $d_X$ over $B$,
  \[
    \sE_{m,n}^Z(X;L):=\sE^{2d_X-m, d_X-n}_Z(X;\omega_{X/B}\otimes L);
  \]
  we extend the definition to arbitrary $X\in \Sm_B$ by taking the sum over the connected components of $X$ and write this also as $\sE^{2d_X-m, d_X-n}_Z(X;\omega_{X/B}\otimes L)$ by considering $d_X$ as a locally constant functor on $X$.

  The pushforward maps for a proper morphism of relative dimension $d$, $f:Y\to X$, closed subsets $W\subset Y$, $Z\subset X$ with $f(W)\subset Z$ and line bundle $L\to X$ are given by the pushforward
  \[
    f_*:\sE^{2d_Y-m,2d_Y-n}_W(Y;\omega_{Y/B}\otimes f^*L)\to 
    \sE^{2d_X-m,2d_X-n}_Z(X;\omega_{X/B}\otimes L).
  \]
\end{example}

\section{Comparison isomorphisms}
\label{sec:Uniq}

We recall the element $\eta\in \Hom_{\SH(B)}(1_B, \rS^{-1,-1}\wedge1_B)$ induced by the map of $B$-schemes $\eta:\A^2\setminus\{0\}\to \P^1$, $\eta(a,b)=(a:b)$. As every $\sE\in \SH(B)$ is a module for $1_B$, we have the map $\times\eta:\sE\to \rS^{-1,-1}\wedge \sE$ for each $x\in \SH(B)$. We say that  {\em $\eta$ acts invertibly on $\sE$} if $\times\eta$ is an isomorphism in $\SH(B)$. 

We consider the following situation: fix an $\SL$-oriented motivic spectrum $\sE\in \SH(B)$. This gives us the twisted cohomology theory $\sE^{*,*}$ underlying the oriented cohomology defined by $\sE$. Let $(\sE^{*,*}, \tilde{\sE}_{*,*})$ be an extension of $\sE^{*,*}$ to an oriented cohomology theory on $\Sm_B$, in other words, we define new pushforward maps
\[
\hat{f}_*:\sE^{*,*}_W(Y, \omega_{Y/B}\otimes f^*L)\to \sE^{*-2d,*-d}_Z(X, \omega_{X/B}\otimes L)
\]
The main result of this section is a comparison theorem. Before stating the result we recall the decomposition of $\SH(B)[1/2]$ into plus and minus parts.

We have the involution $\tau:1_B\to 1_B$ induced by the symmetry isomorphism $\tau:\P^1\wedge\P^1\to\P^1\wedge\P^1$.   In $\SH(B)[1/2]$, this gives us the idempotents $(\id+\tau)/2$, $(\id-\tau)/2$, and so decomposes 
$\SH(B)[1/2]$ into +1 and -1 ``eigenspaces'' for $\tau$:
\[
\SH(B)[1/2]=\SH(B)^+\times\SH(B)^-
\]
We decompose $\sE\in \SH(B)[1/2]$ as $\sE=\sE_+\oplus \sE_-$. 

\begin{theorem}\label{thm:SLUniqueness} Suppose the pushforward maps
\[
f_*, \hat{f}_*:\sE^{*,*}_W(Y, \omega_{Y/B}\otimes f^*L)\to \sE^{*-2d,*-d}_Z(X, \omega_{X/B}\otimes L)
\]
agree for $W$, $Z$,  $L$, $X=V$ a vector bundle over $Y$ and $f:Y\to V$ the zero-section. Suppose in addition that one of the following conditions holds:
\begin{enumerate}
\item the $\SL$-orientation of $\sE$ extends to a $\GL$-orientation;
\item $\eta$ acts invertibly on $\sE$;
\item $2$ acts invertibly on $\sE$ and  $\sE_+^{-1,0}(U)=0$ for affine $U$ in $\Sm_B$.
\end{enumerate}
Then $f_*=\hat{f}_*$ for all $X, Y, Z, W, L, f$ for which the pushforward is defined.
\end{theorem}

\begin{proof} By the standard argument of deformation to the normal cone, it follows that $f_*=\hat{f}_*$ for all $f:Y\to X$ a closed immersion, $Z, W, L$.  As every proper map in $\Sm_B$ is projective, $f$ admits a factorization $f=p\circ i$, with $i:Y\to X\times_B\P^N$ a closed immersion and $p:X\times_B\P^N\to X$ the projection. By functoriality of the pushforward maps, it suffices to check that $p_*=\hat{p}_*$. 

In case (i), this follows from the uniqueness assertion in \cite[Theorem 2.5(i)]{PaninII}. Indeed, the cohomology theory associated to a $\GL$-oriented motivic spectrum $\sE$ satisfies the axioms of Panin-Smirnov and the associated Thom isomorphisms give rise to an ``orientation'' in the sense of \cite[Definition 1.9]{PaninII}, so we may apply the results cited. We note that in \cite{PaninII} the base-scheme is $\Spec k$, $k$ a field, so loc. cit. does not immediately apply to our setting of a more general base-scheme; we say a few words about the extension of this result to our base-scheme $B$. As a proper map $f:Y\to X$ in $\Sm_B$ is projective, one factors $f$ as $f=p\circ i$, with $i:Y\to \P^n_X$ a closed immersion and $p:\P^n_X\to X$ the projection. The uniqueness for  a closed immersion in $\Sm_B$ reduces to the case of the zero-section of a vector bundle by the usual method of deformation to the normal bundle, and as the pushforward by the zero-section of our two theories are the same by assumption, we have agreement in the case of a closed immersion. For the projection $p$, the proof of \cite[Theorem 2.5(i)]{PaninII} relies on \cite[Theorem 1.1.9]{PaninRR}, where for $p$, using the projective bundle formula,  the key point is to show that both pushforwards have the same value on the unit $1_{\P^n_X}\in\sE^{0,0}(\P^n_X)$. The proof of this relies on the formula for the pushforward of $1_{\P^n_X}$ under the diagonal $\Delta_{\P^n_X}:\P^n_X\to \P^n_X\times_X\P^n_X$ given by \cite[Lemma 1.9.4]{PaninRR}. As $\Delta_{\P^n_X}$ is a closed immersion, the two pushforwards under $\Delta_{\P^n_X}$ agree, and the proof of the formula in \cite[Lemma 1.9.4]{PaninRR} uses only formal properties of pushforward and pullback as expressed in our axioms, plus the projective bundle formula. This latter in turn relies only on properties of the Thom class of $\sO(-1)$ and localization with respect to $\A^m_X\subset \P^m_X$, 
 and thus we may use \cite[Lemma 1.9.4]{PaninRR} in our more general setting. The argument that the pushforward of $1_{\P^n_X}$ under $p$ can be recovered from the formula for the pushforward of $1_{\P^n_X}$ under $\Delta_{\P^n_X}$ is elementary and formal, and only uses the restriction of the two theories to $\SmL_X$, and not the fact that these restrictions come from theories over $k$. Thus, the argument used in the proof of \cite[Lemma 1.9.4]{PaninRR} may be used to prove our result in case (i).

In case (ii), we use Lemma~\ref{lem:Case(ii)} below. Indeed, if $N$ is odd, we may apply the closed immersion $X\times_B\P^N\to X\times_B\P^{N+1}$ as a hyperplane, so we reduce to the case $N$ even, in which case both $p_*$ and $\hat{p}_*$ are inverse to the map $i_*$, where $i:X\to X\times_B\P^N$ is the section associated to the point $(1:0:\ldots:0)$ of $\P^N$.

In case (iii) we may work in the category $\SH(B)[1/2]$.  We decompose $\sE\in \SH(B)[1/2]$ as $\sE=\sE_+\oplus \sE_-$ and similarly decompose the pushforward maps $f_*$ and $\hat{f}_*$. By Lemma~\ref{lem:EtaInv}, $\eta$ acts invertibly on $\SH(B)^-$ and the projection of $\eta$ to $\SH(B)^+$ is zero. By Lemma~\ref{lem:Case(iii)} below, the  $\SL$-orientation of $\sE$ induces an $\SL$-orientation on the projection $\sE^+$ that extends to a $\GL$-orientation. By (i), this implies that $f_*^+=\hat{f}_*^+$. By (ii), $f_*^-=\hat{f}_*^-$, so $f_*=\hat{f}_*$.
\end{proof}

\begin{lemma}[\hbox{\cite[Theorem 1]{AnanSL}}]\label{lem:Case(ii)} Let $\sE\in \SH(B)$ be an $\SL$-oriented motivic spectrum on which $\eta$ acts invertibly. Let $0\in \P^N(\Z)$ be the point $(1:0\ldots:0)$. For $X\in \Sm_B$, $L\to X$ a line bundle and $Z\subset X$ a closed subset, the pushforward map
\[
i_*:\sE^{*-2N,*-N}_Z(X, \omega_{X/B}\otimes L)\to  \sE^{*,*}_{p^{-1}(Z)}(X\times_B\P^N, \omega_{\P^N/B}\otimes p^*L)
\]
is an isomorphism.
\end{lemma}
\begin{proof} Using a Mayer-Vietoris sequence, we see that the statement is local on $X$ for the Zariski topology, so we may assume that $L=\sO_X$. If we prove the statement for the pair $(X,X)$ and $(X\setminus Z, X\setminus Z)$ the local cohomology sequence gives the result for $(X,Z)$, thus we may assume that $Z=X$, and we reduce to showing that
\[
i_*:\sE^{*-2N,*-N}(X, \omega_{X/B})\to  \sE^{*,*}(X\times_B\P^N, \omega_{\P^N/B})
\]
is an isomorphism.   

This is \cite[Theorem 4.6]{AnanSL} in case $B=\Spec k$, $k$ a field. The proof over a general base-scheme is essentially the same, we say a few words about this generalization.  Most of the results that are used in the proof of loc. cit. are are already proved in the required generality here, for example, the Thom isomorphism \eqref{eqn:CanonThom} of Construction~\ref{const:CanonThom} generalizes Ananyevskiy's construction \cite[Corollary 1]{AnanSL} from $B=\Spec k$ to general $B$. The proof of \cite[Theorem 4.6]{AnanSL} relies also on \cite[Lemma 4.1]{AnanSL}, which in our setting reduces to the fact that for $X\in \Sch_B$, and $u\in \Gamma(X,\sO_X^\times)$ a unit, the automorphism of $X\times \P^1$ sending $(x,(t_0: t_1))$ to $(x, (ut_0, u^{-1}t_1)$ induces the identity on $X_+\wedge\P^1/X$ in $\H_\bul(X)$. This follows by identifying $\P^1_X$ with $\P(\A^2_X)$ and noting that the diagonal matrix with entries $u, u^{-1}$ is an elementary matrix in $\GL_2(\Gamma(X,\sO_X))$. 
\end{proof}

\begin{lemma}\label{lem:Case(iii)} Suppose that $\sE\in \SH(B)$ is $\SL$ oriented and that $\sE^{-1,0}(U)=0$ for all affine $U$ in $\Sm_B$. Then the induced $\SL$ orientation on $\sE_+\in \SH(B)_+$ extends to a $\GL$ orientation.
\end{lemma}

\begin{proof} Let $u\in \Gamma(X, \sO_X^\times)$ be a unit on some $X\in \Sm_B$. Then the map
\[
\times u:X\times_B\P^1\to X\times_B\P^1;\quad (x, [t_0:t_1])\mapsto  (x, [ut_0:t_1])
\]
induces the identity on $\rS^{2,1}\wedge X/B$ in $\SH(B)_+$. Indeed, let $[u]:X/B\to X/B\wedge\G_m$ be the map induced by $u:X\to \G_m$. The argument given by Morel \cite[6.3.4]{MorelICTP}, that
\[
\times u/B=\id+\eta[u]
\]
in case $B=\Spec k$, $k$ a field, is perfectly valid over a general base-scheme: this only uses the fact that for $\sX$ and $\sY$ pointed spaces over $B$, one has
\[
\Sigma^\infty_{\rS^1}\sX\times_B\sY\simeq \Sigma^\infty_{\rS^1}\sX\oplus \Sigma^\infty_{\rS^1}\sY   \oplus \Sigma^\infty_{\rS^1}(\sX\wedge \sY)
\]
and that the map $\times _u:\rS^1\wedge\G_m\wedge X_+\to \rS^1\wedge \G_m\wedge X_+$ is the $\rS^1$-suspension of the composition
\[
\rS^1\wedge\G_m\wedge X_+\xr{\id\wedge u}\rS^1\wedge(\G_m\times\G_m)\wedge X_+
\xr{\id\wedge \mu\wedge\id}\rS^1\wedge \G_m\wedge X_+
\]
where $\mu:\G_m\times\G_m\to \G_m$ is the multiplication.
As $\eta$ goes to zero in $\SH(B)_+$, it follows that $\times u/B=\id$ in $\SH(B)_+$.

Now take $g\in \Gamma(X, \GL_n(\sO_X))$, let $u=\det g$,  let $m_u\in \Gamma(X, \GL_n(\sO_X))$ be the diagonal  matrix with entries $u, 1,\ldots, 1$ and let $h=m_u^{-1}\cdot g\in \Gamma(X,\SL_n(\sO_X))$. We have 
\[
\Th_X(\sO_X^n)=(\P^1)^{\wedge n}\wedge X_+.
\]
Since $\sE$ is $\SL$-oriented, the map $\Th(h):\Th_X(\sO_X^n)\to \Th_X(\sO_X^n)$ induces the identity on $\sE^{**}$ and thus 
\[
\Th(g)^*=\Th(m_u)^*:\sE^{*,*}_+(\Th_X(\sO_X^n))\to \sE^{*,*}_+(\Th_X(\sO_X^n))
\]
But as  $\Th(m_u)=(\times u)\wedge \id$,  our previous computation shows that $\Th(m_u)^*=\id$.

Now let $V\to X$ be a rank $r$ vector bundle on some $X\in \Sm_B$, choose a trivializing affine open cover $\sU=\{U_i\}$ of $X$ and let $\phi_i:V_{|U_i}\to U_i\times \A^r$ be a local framing. We have the suspension isomorphism
\[
\Th(V_{U_i})\simeq \Th(U_i\times \A^r)=\Sigma_\T^r U_{i+}
\]
giving the isomorphism
\[
\theta_i:\sE_+^{a,b}(U_i)\to \sE^{2r+a, r+b}_{+0_{V_{|U_i}}}(V_{|U_i}).
\]
Since $\GL_r(\sO_{U_i})$ acts trivially on $\sE_+^{**}(\Th(U_i\times \A^r))$, the isomorphism $\theta_i$ is independent of the choice of framing $\phi_i$. 
In addition, the assumption $\sE^{-1,0}(U_i\cap U_j)=0$ implies 
\[
\sE^{2r-1, r}_{+0_{V_{|U_i\cap U_j}}}(V_{|U_i\cap U_j})=0
\]
for all $i,j$. By Mayer-Vietoris, the sections
\[
\theta_i(1_{U_i})\in \sE^{2r, r}_{+0_{V_{|U_i}}}(V_{|U_i})
\]
uniquely extend to an element
\[
\theta_V\in \sE^{2r, r}_{+0_V}(V)
\]

The independence of the $\theta_i$ on the choice of framing and the uniqueness of the extension readily implies the functoriality of $\theta_V$ and similarly implies the product formula $\theta_{V\oplus W}=p_1^*\theta_V\cup p_2^*\theta_W$. By construction, $\theta_V$ is the suspension of the unit over $U_i$, another application of   independence  of the choice of framing and the uniqueness of the extension shows that this is the case over every open subset $U\subset X$ for which $V_{|U}$ is the trivial bundle. Finally, the  independence and uniqueness shows that $V\mapsto \theta_V$ is an extension of the $\SL$ orientation on $\sE_+$ induced by that of $\sE$.
\end{proof} 

\begin{lemma} For $u\in \Gamma(X,\sO_X^\times)$ we have
\[
[u]\eta=\eta[u]:\Sigma^\infty_\T  X_+\to \Sigma^\infty_\T  X_+ 
\]
\end{lemma}

\begin{proof} We use the decomposition
\[
\Sigma^\infty_\T  X_+\wedge\G_m\times\G_m=\Sigma^\infty_\T  X_+\wedge\G_m\oplus \Sigma^\infty_\T  X_+\wedge\G_m\oplus \Sigma^\infty_\T  X_+\wedge\G_m\wedge\G_m
\]
Via this, $\eta$ is the map
\[
[s]\wedge [t]\mapsto [st]-[s]-[t]
\]
so $\eta[u]$ sends $[t]$ to $[ut]-[u]-[t]$ and $\id_{\G_m}\wedge\eta[u]$ sends 
$[s]\wedge [t]$ to $[s]\wedge[ut]-[s]\wedge[u]-[s]\wedge[t]$, so $[u]\eta$ is given by
\[
[s]\wedge [t]\mapsto [st]-[s]-[t]\mapsto [u]\wedge[st]-[u]\wedge[s]-[u]\wedge[t].
\]
We have the automorphism $\xi$ of  $\G_m^{\wedge 3}$ sending $[u]\wedge[s]\wedge [t]$ to
$[s]\wedge [t]\wedge [u]$. We have the isomorphism in $\H_\bul(B)$, $\Sigma_{S_1}^6\G_m^{\wedge 3}\simeq \A^3/\A^3\setminus\{0\}$, via which $\Sigma_{S_1}^6\xi$ is induced by the linear map $(u,s,t)\mapsto (s, t, u)$. As this latter linear map has matrix in the standard basis a product of elementary matrices, $\Sigma_{\rS^1}^6\xi$ is $\A^1$-homotopic to the identity, so after stabilizing,  $\id_{\G_m}\wedge\eta[u]$ is the map
\[
[s]\wedge [t]\mapsto [s]\wedge [t]\wedge[u]\mapsto[u]\wedge [s]\wedge [t]\mapsto
[u]\wedge[st]-[u]\wedge[s]-[u]\wedge[t]=[u]\eta([s]\wedge[t]).
\]
\end{proof}

\begin{lemma}\label{lem:EtaInv} The projection  $\eta_-$ of $\eta$ to $\SH(B)_-$ is an isomorphism and the projection $\eta_+$ of $\eta$ to $\SH(B)_+$ is zero.
\end{lemma}

\begin{proof} Morel proves this in \cite[\S 6]{MorelICTP} in the case of a field, but the proof works in general.   In some detail, the map $\tau$ is the map on $\A^2/(\A^2\setminus\{0\})$ induced by the linear map $(x, y)\mapsto (y,x)$. The matrix identity
\[
\begin{pmatrix}
0&1\\
1&0
\end{pmatrix}=
 \begin{pmatrix}
1&1\\
0&1
\end{pmatrix}\cdot\begin{pmatrix}
1&0\\
-1&1
\end{pmatrix}\cdot \begin{pmatrix}
1&1\\
0&1
\end{pmatrix}\cdot
\begin{pmatrix}
-1&0\\
0&1
\end{pmatrix}
\]
shows that the maps $(x, y)\to (y,x)$ and $(x,y)\mapsto (-x, y)$ are $\A^1$-homotopic. By the arguments in Lemma~\ref{lem:Case(iii)}, this latter map induces the map $1+\eta[-1]=1+[-1]\eta$ in $\SH(B)$, giving the identity
\[
(1+\eta[-1])_-=(1+[-1]\eta)_-=-\id\Rightarrow \eta\cdot (-[-1]/2)= (-[-1]/2)\cdot \eta=\id_{\SH(B)_-}
\]

For $\eta_+$, the projector to $\SH(B)_+$ is given by the idempotent $(1/2)(\tau+1)=(1/2)(2+\eta[-1])$, so $\eta_+=(1/2)\eta\cdot(2+\eta[-1])$. Since the map $\tau:\P^1\wedge\P^1\to \P^1\wedge\P^1$ is $1+\eta[-1]$ and $\P^1=\rS^1\wedge\G_m$, the symmetry $\epsilon:\G_m\wedge\G_m\to \G_m\wedge\G_m$ is $-(1+\eta[-1])$. From our formula for $\eta([s]\wedge[t])$ we see that $\eta\epsilon=\eta$ which gives $\eta\cdot(2+\eta[-1])=0$. 
\end{proof}

\section{Applications}
\label{sec:App}

In this section, we apply the motivic Gau{\ss}-Bonnet formula of \S\ref{sec:GaussBonnet} and the comparison results of \S\ref{sec:Uniq} to various specific $\SL$-oriented cohomology theories, and thereby make computations of the motivic Euler characteristic $\chi(X/B)$ in different contexts.

\subsection{Motivic cohomology and cohomology of the Milnor K-theory sheaves}

We work over the base-scheme $B=\Spec k$, with $k$ a perfect field. In $\SH(k)$ we have the motivic cohomology spectrum   $\H\Z$ representing Voevodsky's motivic cohomology (see e.g. \cite[\S 6.2]{LevHC} for a construction valid in arbitrary characteristic). By \cite{VoevodskyChow}, there is a natural isomorphism
\[
\H\Z^{a,b}(X)\simeq \CH^b(X, 2b-a)
\]
for $X\in \Sm_B$, where $\CH^b(X, 2b-a)$ is Bloch's higher Chow group \cite{Bloch}.

$\H\Z$ admits a localization sequence: for $i:Z\to X$ a closed immersion of codimension $d$ in $\Sm_k$, there is a canonical isomorphism
\[
\H\Z^{a,b}_Z(X)\simeq \H\Z^{a-2d, b-d}(Z)
\]
See for example \cite{BlochML}. In particular, for $p:V\to X$ a rank $r$ vector bundle over $X\in \Sm_k$, we have the isomorphism
\[
\H\Z^{2r, r}_{0_V}(V)\simeq \H\Z^{0,0}(X)
\]
which gives us Thom classes $\vartheta_V^{\H\Z}\in \H\Z^{2r, r}_{0_V}(V)$ corresponding to the unit $1^{\H\Z}_X\in \H\Z^{0,0}(X)$. Thus $\H\Z$ is a $\GL$-oriented motivic spectrum.

Let $X$ be a smooth projective $k$-scheme of dimension $n$ over $k$. For a class $x\in \H\Z^{2n, n}(X)$, the isomorphism $\H\Z^{2n,n}(X)\simeq \CH^n(X,0)=\CH^n(X)$ allows one to represent $x$ as the class of a 0-cycle $\tilde{x}=\sum_in_i p_i$, with the $p_i$ closed points of $X$. One has the degree $\deg_k(p_i):=[k(p_i):k]$ and extending by linearity gives the degree $\deg_k(\tilde{x})$, which one shows passes to rational equivalence to define a degree map
\[
\deg_k: \H\Z^{2n, n}(X)\to \CH^0(\Spec k)=\Z.
\]

As a $\GL$-oriented theory, $\H\Z$ has Chern classes for vector bundles: $c_r(V)\in \H\Z^{2r, r}(X)$ for $V\to X$ a vector bundle over some $X\in \Sm_k$ and $r \ge 0$.

\begin{theorem}\label{thm:EulerChow}
  Let $X\in \Sm_k$ be projective of dimension $d_X$. Then
  \[
    u^{\H\Z}(\chi(X/k))=\deg_k(c_{d_X}(T_{X/k})).
  \]
\end{theorem}

\begin{proof}  One has well-defined pushforward maps on $\CH^*(-,*)$ for projective morphisms (see e.g. \cite[Proposition 1.3]{Bloch}). Via the isomorphism $ \H\Z^{a,b}(X)\simeq \CH^b(X, 2b-a)$ \cite{VoevodskyChow}, this gives pushforward maps $\hat{f}_*$ on $\H\Z^{*,*}$ for $f:Y\to X$ a projective morphism in $\Sm_B$ (see \cite{Bloch} for details),  making $(X, Z)\mapsto \H\Z^{*,*}_Z(X)$ a $\GL$-oriented cohomology theory on $\Sm_k$. In addition, for $\pi_X:X\to \Spec k$ in $\Sm_k$ projective of dimension $n$,  the map  $\hat{\pi}_{X*}:\H\Z^{2n, n}(X)\to \CH^0(\Spec k)=\Z$ is $\deg_k$, and for $i:Y\to X$ a closed immersion, the map $\hat{i}_*$ is given by the localization theorem, which readily implies that $\hat{i}_*=i_*$. By our comparison theorem \ref{thm:SLUniqueness}, which here really reduces to the theorem of Panin-Smirnov, it follows that $\hat{f}_*=f_*$ for all projective $f$.

Finally, one has $c_{d_X}=s^*s_*(1^{\H\Z}_X)=e^{\H\Z}(V)$ (\cite[Corollary 6.3]{Fulton}), so applying the motivic Gau{\ss}-Bonnet theorem \ref{thm:GaussBonnet} gives the statement.
\end{proof}

One can obtain the same result by using the cohomology of the Milnor K-theory sheaves as a bigraded cohomology theory. The homotopy t-structure on $\SH(k)$ has heart the abelian category of homotopy modules $\HM(k)$  (see \cite[\S 5.2]{MorelICTP} and \cite{MorelA1} for details); we let $\H_0:\SH(k)\to \HM(k)$ be the associated functor.  The fact that $\H\Z^{n,n}(\Spec F) \simeq \K^\M_n(F)$ for $F$ a field \cite{NS, Totaro} says that $\H_0\H\Z$ is canonically isomorphic to the homotopy module $(\K^\M_n)_n$, which is in fact a cycle module in the sense of Rost \cite{Rost}. This gives us the isomorphism
\[
\H_0\H\Z^{a,b}(X)\simeq \H^a(X, \sK^\M_b).
\]
The isomorphism $H^n(X, \sK^\M_n)\simeq \CH^n(X)$  (a special case of Rost's formula for the Chow groups of a cycle module, \cite[Corollary 6.5]{Rost})  gives us as above Thom classes $\vartheta^{\K^\M_*}(V)\in \H_0\H\Z^{2r,r}_{0_V}(V)$, giving $\H_0\H\Z$ a $\GL$-orientation. As for $\H\Z^{*,*}$, one has explicitly defined pushforward maps on $\H^*(-, \sK^M_*)$ which give $\H_0\H\Z^{*,*}$ the structure of a $\GL$-oriented cohomology theory on $\Sm_k$ and for which the pushforward map for the zero-section of a vector bundle is given by the Thom isomorphism. Since the pushforward on $\H^n(X, \sK^\M_n)$ agrees with the classical pushforward on $\CH^n$, we deduce the following using the same proof as for Theorem~\ref{thm:EulerChow}.

\begin{theorem}\label{thm:EulerMilnor}
  Let $X\in \Sm_k$ be projective of dimension $d_X$. Then
  \[
    u^{\H_0\H\Z}(\chi(X/k))=\deg_k(c_{d_X}(T_{X/k})) \text{ in }\CH^0(\Spec k)=\Z.
  \]
\end{theorem}

\subsection{Algebraic K-theory}
\label{sec:K}

We now let $B$ be any regular separated base-scheme of finite Krull dimension. Algebraic K-theory on $\Sm_B$ is represented by the motivic commutative ring spectrum $\KGL\in \SH(B)$ (see \cite[\S 6.2]{VoevICM}). Just as for $\H\Z$, the  purity theorem
\[
\KGL_Z^{a,b}(X)\simeq \KGL^{a-2c, b-c}(Z)
\]
for $i:Z\to X$ a closed immersion of codimension $d$ in $\Sm_B$  (a consequence of Quillen's localization sequence for algebraic K-theory \cite[\S 7, Proposition 3.2]{QuillenKThy}) gives Thom class $\vartheta^\KGL(V)\in \KGL^{2r,r}_{0_V}(V)$ for $V\to X$ a rank $r$ vector bundle over $X\in \Sm_B$, and makes $\KGL$ a $\GL$-oriented motivic spectrum. 

Explicitly, $\KGL$ represents Quillen K-theory on $\Sm_B$ via $\KGL^{a,b}\simeq \K_{2b-a}$ and the Thom class for a rank $r$ vector bundle $p:V\to X$ is represented by the Koszul complex $\Kos_V(p^*V^\vee, \can)$. Here $\can:p^*V^\vee\to \sO_V$ is the dual of the tautological section $\sO_V\to p^*V$ and $\Kos_V(p^*V^\vee, \can)$ is the complex whose terms are given by
\[
\Kos_V(p^*V^\vee, \can)^{-r}=\Lambda^rp^*V^\vee
\]
and whose differential $\Lambda^rp^*V^\vee\to \Lambda^{r-1}p^*V^\vee$ is given with respect to a local framing of $V^\vee$ by
\[
d(e_{i_1}\wedge\ldots\wedge e_{i_r})=\sum_{j=1}^r(-1)^{j-1}\can(e_{i_j})\cdot 
e_{i_1}\wedge\ldots\wedge \widehat{e_{i_j}}\wedge\ldots\wedge e_{i_r}.
\]
This complex is a locally free resolution of $s_*(\sO_X)$, where $s:X\to V$ is the zero-section. Thus, by the identification of $\KGL^{2r, r}_{0_V}(V)$ with the Grothendieck group of the triangulated category of perfect complexes on $V$ with support contained in $0_V$, $\Kos_V(p^*V^\vee, \can)$ gives rise to a class $[\Kos_V(p^*V^\vee, \can)]\in \KGL^{2r, r}_{0_V}(V)$ which maps to $1_X$ under the purity isomorphism $\KGL^{2r, r}_{0_V}(V)\simeq \KGL^{0,0}(X)$, so that we indeed have $[\Kos_V(p^*V^\vee, \can)] = \vartheta^\KGL(V)$.

Just as for motivic cohomology, one has explicit pushforward maps in K-theory given by Quillen's localization  and devissage theorems identifying, for $X \in \Sm_B$ and $Z \subseteq X$ a closed subscheme, the K-theory with support $\K^Z(X)$ with the K-theory of the abelian category  of coherent sheaves $\Coh_Z$ on $Z$, denoted $\rG(Z)$. For a projective morphism $f:Y\to X$, one has the pushforward map $\hat{f}_*:\rG(Y)\to \rG(X)$ defined by using a suitable subcategory of $\Coh_Y$ on which $f_*$ is exact. On $\K_0$, this recovers the usual formula
\[
\hat{f}_*([\sF])=\sum_{j=0}^{\dim Y}(-1)^j[\rR^jf_*(\sF)]
\]
for $\sF\in \Coh_Y$. Via the isomorphisms $\KGL_Z^{a,b}(X)\simeq \rG_{2b-a}(Z)$, this gives pushforward maps $\hat{f}_*$ for $\KGL^{*,*}$, defining a $\GL$-oriented cohomology theory on $\Sm_B$.

For $s:X\to V$ the zero-section of a vector bundle, $\hat{s}_*$ agrees with the pushforward  $s_*$ using the Thom isomorphism/localization theorem, hence by our comparison theorem (again really the theorem of Panin-Smirnov), we have $\hat{f}_*=f_*$ for all projective $f$.

\begin{theorem}\label{thm:EulerKThy}
  Let $\pi_X:X\to B$ be a smooth projective morphism with $B$ a regular separated scheme of finite Krull dimension. Then
  \[
    u^\KGL(\chi(X/B))=\sum_{j=0}^{\dim_BX}\sum_{i=0}^{\dim_BX} (-1)^{i+j}[\rR^j\pi_{X*}\Omega_{X/B}^i]\in \K_0(B)=\KGL^{0,0}(B).
  \]
\end{theorem}

\begin{proof}
  Let $p:T_{X/B}\to X$ denote the relative tangent bundle and let $s:X\to T_{X/B}$ denote the zero-section. We have
  \[
    e^\KGL(T_{X/B})=s^*(\th(T_{X/B}))=s^*(\Kos_{T_{X/B}}(p^*T_{X/B}^\vee, \can)).
  \]
  Since $T_{X/B}^\vee=\Omega_{X/B}$, and $s^*(\can)$ is the zero-map, it follows that, in $\K_0(X)$, 
  \[
    s^*(\Kos_{T_{X/B}}(p^*T_{X/B}^\vee, \can))=\sum_{i=0}^{\dim_BX}(-1)^i[\Omega_{X/B}^i],
  \]
  and thus
  \[
    \pi_{X*}(e^\KGL(T_{X/B}))=\sum_{j=0}^{\dim_BX}\sum_{i=0}^{\dim_BX} (-1)^{i+j}[\rR^j\pi_{X*}\Omega_{X/B}^i].
  \]
  We conclude by applying the motivic Gau{\ss}-Bonnet theorem.
\end{proof}

\subsection{Milnor-Witt cohomology and Chow-Witt groups}

In this case, we again work over a perfect base-field $k$. The Milnor-Witt  sheaves  $\sK^\MW_*$ constructed by Morel (\cite[\S 6]{MorelICTP}, \cite[Chapter 3]{MorelA1}) give rise to an $\SL$-oriented theory as follows. Morel describes an isomorphism of $\sK^\MW_0$ with the sheafification $\sGW$ of the Grothendieck-Witt rings\footnote{\cite[Lemma 3.10]{MorelA1} defines an isomorphism $\GW(F)\to \K^\MW(F)$, $F$ a field. \cite[\S 3.2]{MorelA1} defines  $\sK^\MW_*$ as an unramified sheaf and it follows from \cite[Theorem A]{PaninOjan} that $\sGW$ is an unramified sheaf. From this it is not difficult to show that the isomorphism  $\GW(F)\to \K^\MW(F)$ for fields extends to an isomorphism of sheaves.}; the map of sheaves of abelian groups $\G_m\to \sGW^\times$ sending a unit $u$ to the one-dimensional form $\<u\>$ allows one to define, for $L\to X$ a line bundle, a twisted version
\[
\sK^\MW_*(L):=\sK^\MW_*\times_{\G_m}L^\times
\]
as a Nisnevich sheaf on $X\in \Sm_k$  (see \cite[pg. 118]{MorelA1} or \cite[\S 1.2]{CalmesFasel}). One may use the Rost-Schmid complex for $\sK^\MW_*(L)$ \cite[Chapter 5]{MorelA1} to compute $\H^*_Z(X, \sK^\MW_*(L))$ for $Z\subseteq X$ a closed subset, which gives a purity theorem: for $i:Z\to X$ a codimension $d$ closed immersion in $\Sm_k$ and $L\to X$ a line bundle, there is a canonical isomorphism
\begin{equation}\label{eqn:MWPurity}
  \H^*_Z(X, \sK^\MW_*(L))\simeq \H^{*-d}(Z, \sK^\MW_{*-d}(i^*L\otimes \det N_i)),
\end{equation}
where $N_i\to Z$ is the normal bundle of $i$. Applying this to the zero-section of a rank $r$ vector bundle $p:V\to X$ gives the isomorphism
\begin{equation}\label{eqn:RSIso}
\H^0(X, \sGW)\simeq \H^r_{0_V}(V, \sK^\MW_r(p^*\det^{-1} V));
\end{equation}
in particular, given an isomorphism $\phi:\det V\to \sO_X$, we obtain a Thom class
\[
\theta_{V, \phi}\in \H^r_{0_V}(V, \sK^\MW_r)
\]
corresponding to the unit section $1_X\in \H^0(X, \sGW)$.

On the other hand, Morel's computation of the 0th graded homotopy sheaf of the sphere spectrum (\cite[Theorem 6.4.1, Remark 6.4.2]{MorelICTP}, \cite[Theorem 6.40]{MorelA1}) gives an identification
\[
\H_0(1_k)\simeq (\sK^\MW_n)_{n\in \Z}
\]
in $\HM(k)$, which then gives the natural isomorphism
\[
\H_0(1_k)^{a+b, b}_Z(X)\simeq \H^a_Z(X, \sK^\MW_b).
\]
This is moreover compatible with twisting by a line bundle $p:L\to X$, on the $\H_0(1_k)$ side using the Thom space construction
\[
\H_0(1_k)^{*, *}_Z(X;L):=\H_0(1_k)^{*+2, *+1}_Z(L)
\]
and on the Milnor-Witt cohomology side using the twisted Milnor-Witt sheaves. To see this, note that the ``untwisted'' isomorphism gives us an isomorphism
\[
\H_0(1_k)^{a+b+2, b+1}_{Z}(\Th(L))\simeq \H_0(1_k)^{a+b+2, b+1}_{0_L\cap p^{-1}(Z)} (L)\simeq \H^{a+1}_{0_L\cap p^{-1}(Z)}(X, \sK^\MW_{b+1}),
\]
so it suffices to identify the right-hand side with $\H^a_Z(X,\sK^\MW_b(L))$. For $Y\in \Sm_k$ and line bundle $M\to Y$, the Rost-Schmid complex for $\sK^\MW_m(M)$ consists in degree $a$ of sums of  terms of twisted Milnor-Witt groups of the form $\K^\MW_{m-a}(k(y); \Lambda^a(\mathfrak{m}_y/\mathfrak{m}_y^2)^\vee\otimes M)$, for $y$ a codimension $a$ point of $Y$ and $\mathfrak{m}_y\subset \sO_{Y,y}$ the maximal ideal. To compute cohomology with supports in $W\subseteq Y$, one restricts to those $y\in W$. If we now take $Y=L$ and $M$ the trivial bundle, with supports in $p^{-1}(Z)\cap 0_L$ and $m=b+1$, and compare with $Y=X$, with supports in $Z$ with $m=b$, the term for $y\in Z$, of codimension $a+1$ on $L$ is $\K^\MW_{b-a}(k(y); \Lambda^a(\mathfrak{m}_y/\mathfrak{m}_y^2)^\vee\otimes L)$ while the term for $y\in Z$, of codimension $a$ on $X$ is $\K^\MW_{b-a}(k(y); \Lambda^a(\mathfrak{m}_y/\mathfrak{m}_y^2)^\vee)$, where $\mathfrak{m}_y$ is the maximal ideal in $\sO_{X,y}$ in both cases. This gives the desired identification
\[
\H^{a+1}_{0_L\cap p^{-1}(Z)}(X, \sK^\MW_{b+1})\simeq \H^{a}_{Z}(X, \sK^\MW_{b}(L)).
\]
The purity isomorphism \eqref{eqn:MWPurity} is a special case of this construction.

The Thom class $\theta_{V, \phi}\in \H^r_{0_V}(V, \sK^\MW_r)$ gives the Thom class
\[
\theta_{V, \phi}\in \H_0(1_k)^{2r, r}_{0_V}(V),
\]
making $\H_0(1_k)$ an $\SL$-oriented theory (see e.g \cite[\S 3.2]{LevEnum}). The  resulting canonical Thom class
\[
\th_V\in \H_0(1_k)^{2r, r}_{0_V}(V;\det^{-1}V)=\H^r_{0_V}(V, \sK^\MW_r(\det^{-1} V))
\]
agrees with the image of $1_X\in \H^0(X, \sGW)$ under the Rost-Schmid isomorphism \eqref{eqn:RSIso}.

Let $\pi_X:X\to \Spec k$ be smooth and projective over $k$ of dimension $d$. Using the  Rost-Schmid complex for the twisted homotopy module  one has generators for $\H^d(X, \sK^\MW_d(\omega_{X/k}))$ as formal sums $\tilde{x}=\sum_i\alpha_i\cdot p_i$, with $\alpha_i\in \GW(k(p_i))$ and $p_i\in X$ closed points. Since $k$ is perfect, the finite extension $k(p_i)/k$ is separable and one can define
\[
\widetilde{\deg}_k(\tilde{x}):=\sum_i\Tr_{k(p_i)/k}\alpha_i\in \GW(k)
\]
where $\Tr_{k(p_i)/k}:\GW(k(p_i))\to \GW(k)$ is the transfer induced by the usual trace map $\Tr_{k(p_i)/k}:k(p_i)\to k$ (see for example \cite[Lemma 2.3]{CalmesFasel} ). It is shown in \cite[\S 3]{CalmesFasel} that this descends to a map
\[
\widetilde{\deg}_k:\H_0^{2d,d}(X;\omega_{X/k})=\H^d(X, \sK^\MW_d(\omega_{X/k}))\to 
\H_0^{0,0}(\Spec k)=\GW(k)
\]
See also \cite[Lemma 5.10]{HoyoisTrace}, which identifies this map with one induced by the Scharlau trace.

The methods of this paper give a new proof of the result given in \cite[Lemma 1.5]{LevEnum}:

\begin{theorem}\label{thm:EulerClassMW}
  Let $k$ be a perfect field of characteristic different from two. For $\pi_X:X\to \Spec k$ smooth and projective over $k$, we have
  \[
    \chi(X/k)=\widetilde{\deg}_k(e^{\H_0(1_k)}(T_{X/k})).
  \]
\end{theorem}

\begin{proof}
  Under Morel's isomorphism $\End_{\SH(k)}(1_X)\simeq \GW(k)$ (\cite[Theorem 6.4.1, Remark 6.4.2]{MorelICTP}, \cite[Theorem 6.40]{MorelA1}) and the isomorphism $\H^0(\Spec k, \sK^\MW_0)\simeq \GW(k)$, the unit map $u^{\H_0(1_k)}:1_k\to \H_0(1_k)$ induces the identity map on $\pi_{0,0}$. Using this, the proof of the claim is essentially the same as the other Gau{\ss}-Bonnet theorems we have discussed, but with a bit of extra work since we are no longer in the GL-oriented case.

  Fasel \cite{FaselCW} has defined pushforward maps 
  \[
    \hat{f}_*:\H^a_W(X, \sK^\MW_b(\omega_{X/k}\otimes f^*L))\to \H^{a-d}_Z(Y, \sK^\MW_{b-d}(L)) 
  \]
  for each projective morphism $f:X\to Y$ in $\Sm_k$ of relative dimension $d$, line bundle $L\to Y$, and closed subsets $Z\subseteq Y$, $W\subseteq X$ with $f(W)\subseteq Z$. In the case of the structure map $\pi_X:X\to \Spec k$, the pushforward $\tilde{\pi}_{X*}:\H^d(X, \sK^\MW_d(\omega_{X/k}))\to \H^0(\Spec k, \sK^\MW_0)=\GW(k)$ is the map $\widetilde{\deg}_k$. 

  For $s:X\to V$ the zero-section of a vector bundle, $\hat{s}_*$ is the Thom isomorphism  $s_*$. Thus, if we pass to the $\eta$-inverted theory, $\H_0(1_X)_\eta:=\H_0(1_k)[\eta^{-1}]$, our comparison theorem \ref{thm:SLUniqueness} says that $\hat{f}_{\eta *}=f_{\eta*}$ for all projective morphisms $f$ in $\Sm_k$. We have $\sK^\MW_*[\eta^{-1}]\simeq \sW$, the sheaf of Witt rings, and the map $\sK^\MW_0=\sGW\to \sK^\MW_*[\eta^{-1}]\simeq \sW$ is the canonical map $q:\sGW\to \sW$ realizing $\sW$ as the quotient of $\sGW$ by the subgroup generated by the hyperbolic form. Thus, applying our motivic Gau{\ss}-Bonnet theorem gives the identity
  \[
    q(\chi(X/k))=q(\widetilde{\deg}_k(e^{\H_0(1_k)}(T_{X/k})))\text{ in }\W(k).
  \]

  To lift this to an equality in $\GW(k)$ and thereby complete the proof, we use that the map
  \[
    (\rnk, q): \sGW\to \Z\times \sW
  \]
  is injective, together with the fact that we can recover the rank by applying $\H_0$ to the unit map $1_k\to \H\Z$ and using Theorem~\ref{thm:EulerMilnor}.
\end{proof}

\subsection{Hermitian K-theory and Witt theory}
\label{sec:KO}

We again let our base-scheme $B$ be a regular noetherian separated base-scheme of finite Krull dimension, but now assume that $2$ invertible on $B$. Our goal in this subsection is to explain how the description of the ``rank''  of $\chi(X/B)$ given by Theorem~\ref{thm:EulerKThy} can be refined to give a formula for $\chi(X/k)$ itself in terms of Hodge cohomology by using hermitian K-theory.

By work of Panin-Walter \cite{PaninWalter}, Schlichting \cite{Schlichting}, and Schlichting-Tripathi \cite{SchlichtingTripathi}, hermitian K-theory $\KO^{[*]}_*(-)$ is represented by a motivic commutative ring spectrum  $\BO\in \SH(B)$ (we use the notation of \cite{AnanWitt}).  Panin-Walter give $\BO$ an $\SL$-orientation.  $\BO$-theory also represents particular cases of Schlichting's Grothendieck-Witt groups \cite{Schlichting}, via functorial isomorphisms 
\[
\BO^{2r, r}(X;L)\simeq \KO^{[r]}_0(X, L):=\GW(\D_\perf(X), L[r], \can),
\]
where $L\to X$ is a line bundle and $\GW(\D_\perf(X), L[r], \can)$ is the Grothendieck-Witt group of {\em $L[r]$-valued symmetric bilinear forms on $\D_\perf(X)$}; we recall a version of the definition here.  

\begin{definition}
  Let  $L\to X$ be a line bundle and let $r \in \Z$. An \emph{$L[r]$-valued symmetric bilinear form} on $C\in \D_\perf(X)$ is a map
  \[
    \phi:C\lotimes C\to L[r]
  \]
  in $\D_\perf(X)$ which satisfies the following conditions.
  \begin{enumerate}
  \item $\phi$ is {\em non-degenerate}: the induced map $C\to \sRHom(C, L[n])$ is an isomorphism in $\D_\perf(X)$.
  \item $\phi$ is {\em symmetric}: $\phi\circ\tau=\phi$, where $\tau:C\lotimes C\to C\lotimes C$ is the commutativity isomorphism.
  \end{enumerate}
  (Note that we are assuming non-degeneracy in the definition but leaving this out of the terminology for the sake of brevity.)
\end{definition}

Similar to the case of algebraic K-theory discussed in \S\ref{sec:K}, for a rank $r$ vector bundle $p:V\to X$, the Thom class $\theta^\BO_V\in \BO^{2r,r}(V; p^*\det^{-1}V)$ is given by the Koszul complex $\Kos(p^*V^\vee, s_{\can}^\vee)$, where the symmetric bilinear form
\[
\phi_V:\Kos(p^*V^\vee, s_{\can}^\vee)\otimes \Kos(p^*V^\vee, s_{\can}^\vee)\to p^*\det^{-1}V[r]=
\Lambda^rV^\vee[r]
\]
is given by the usual exterior product
\[
-\wedge-: \Lambda^iV^\vee\otimes\Lambda^{r-i}V^\vee\to \Lambda^rV^\vee.
\]

Moreover, there are isomorphisms for $i<0$
\[
\BO^{2r-i, r}(X;L)\simeq \W^{r-i}(\D_\perf(X), L[r], \can)
\]
where $\W^{r-i}(\D_\perf(X), L[r], \can)$ is Balmer's triangulated Witt group. In the case $B = \Spec k$ for $k$ a field of characteristic different from two, Ananyevskiy \cite[Theorem 6.5]{AnanWitt} shows that this isomorphism induces an isomorphism of $\eta$-inverted hermitian K-theory with Witt-theory,
\[
\BO[\eta^{-1}]^{*,*}\simeq \W^*[\eta,\eta^{-1}],
\]
where one gives $\eta$ bidegree $(-1,-1)$ and an element $\alpha\eta^n$ with $\alpha\in \W^m$ has bidegree $(m-n,-n)$; the same proof works over out general base $B$ (with assumptions as at the beginning of this subsection). 

For $f:Y\to X$ a proper map of relative dimension $d_f$ in $\Sm_B$ and $L$ a line bundle on $X$, we follow Calm\`es-Hornbostel \cite{CalmesHornbostel} in defining a pushforward map 
\[
\hat{f}_*:\BO^{2r,r}(Y, \omega_{Y/B}\otimes f^*L)\to \BO^{2r-2d_f,r-d_f}(X, \omega_{X/B}\otimes L)
\]
by Grothendieck-Serre duality. In op. cit., this is worked out for the $\eta$-inverted theory $\BO_\eta$ when the base is a field; however, the same construction works for $\BO$ over the general base-scheme $B$ and goes as follows. For $r \ge 0$, given an $L[r]$-valued symmetric bilinear form $\phi:C\lotimes  C\to \omega_{Y/B}\otimes f^*L[r]$, we have the corresponding isomorphism
\[
\tilde{\phi}:C\to \sRHom(C, \omega_{Y/B}\otimes f^*L[r]) \simeq
\sRHom(C, \omega_{Y/X}\otimes f^*(\omega_{X/B}\otimes L[r]))
\]
Grothendieck-Serre duality gives the isomorphism
\[
\rR f_*\sRHom(C, \omega_{Y/X}\otimes f^*(\omega_{X/B}\otimes L[r]))\xymatrix{\ar[r]^{\psi}_\sim&}
\sRHom(\rR f_*C, \omega_{X/B}\otimes L[r-d_f]).
\]
Composing these, we obtain the isomorphism
\[
\psi\circ\tilde{\phi}:\rR f_*C\to \sRHom(\rR f_*C, \omega_{X/B}\otimes L[r-d_f]),
\]
corresponding to the (nondegenerate) bilinear form
\[
\rR f_*(\phi):\rR f_*C\lotimes  \rR f_*C\to  \omega_{X/B}\otimes L[r-d_f],
\]
which one can show is symmetric. We explicitly define the above pushforward map by setting $\hat{f}_*(C, \phi) := (\rR f_*C, \rR f_*(\phi))$.

Applying this in the situation that $f=\pi_X:X\to B$ is a smooth and proper $B$-scheme of relative dimension $d_X$, we may obtain the formula
\[
 \hat{\pi}_{X*}(e^\BO(T_{X/B}))=(\oplus_{i,j=0}^{\dim_BX}\rR^i\pi_{X*}\Omega_{X/B}^j)[j-i], \Tr),
\]
where 
\[
\Tr:(\oplus_{i,j=0}^{\dim_kX}\rR^i\pi_{X*}\Omega_{X/B}^j)[j-i])\otimes(\oplus_{i,j=0}^{\dim_kX}\rR^i\pi_{X*}\Omega_{X/B}^j)[j-i])\to \sO_B
\]
is the symmetric bilinear form in $\D_\perf(B)$ determined by the pairings
\[
(\rR^i\pi_{X*}\Omega_{X/B}^j)\otimes (\rR^{d_X-i}\pi_{X*}\Omega_{X/B}^{d_X-j})\xr{\cup}
\rR^{d_X}\pi_{X*}\Omega_{X/B}^{d_X} \xr{\Tr}\sO_B.
\]
Indeed, if $s:X\to T_{X/B}$ denotes the zero-section, we have
\[
e^\BO(T_{X/B})=s^*(\Kos(T_{X/B}), \phi)=(\oplus_{j=0}^{d_X}\Omega^j_{X/B}[j],s^*\phi)
\]
with $s^*\phi$ determined by the product maps
\[
\Omega^j_{X/B}[j]\otimes \Omega^{d_X-j}_{X/B}[d_X-j]\to \omega_{X/B}[d_X],
\]
and thus $\hat{\pi}_{X*}(e^\BO(T_{X/B}))$ is $\oplus_{i,j=0}^{\dim_kX}\rR^i\pi_{X*}\Omega_{X/B}^j)[j-i]$ with the symmetric bilinear form $\Tr$ as described above. 

Passing to the $\eta$-inverted theory $\BO_\eta$, our comparison theorem~\ref{thm:SLUniqueness} gives 
\[
q\circ \hat{\pi}_{X*}=q\circ \pi_{X*}
\]
as maps $\BO_\eta^{2d_X, d_X}(X,\omega_{X/B})\to \BO_\eta(B)$. We check that the conditions of the comparison theorem hold just as we did for algebraic K-theory. Firstly, as mentioned above, the $\SL$-orientation for $\BO$ defined by Panin-Walter can be described as follows: the Thom class for an oriented vector bundle ($p:V\to X$, $\rho:\sO_X\xr{\sim}\det V$) is given by the Koszul complex $\Kos(p^*V^\vee, s_{\can}^\vee)$ equipped with the symmetric bilinear form $\phi_V$ defined by the product in the exterior algebra followed by the isomorphism $p^*\rho^\vee:p^*\det^{-1}V\to \sO_V$. On the other hand, the Calm\`es-Hornbostel pushforward for the zero-section $s:X\to V$ of a rank $r$ vector bundle $p:V\to X$ with isomorphism $\rho:\sO_X\xr{\sim} \det V$ is as described above, sending a symmetric bilinear form $\psi:C\lotimes C\to \omega_{X/B} \otimes s^*L[n]$ to the symmetric bilinear form
\[
\rR s_*(\psi): \rR s_*C \lotimes \rR s_*C\to \omega_{V/B} \otimes L[n+r].
\]
Since $s$ is finite, $\rR s_*C \simeq C$, which in $\D_\perf(V)$ is canonically isomorphic to $p^*C\otimes_{\sO_V}\Kos(p^*V^\vee, s_{\can}^\vee)$. We may thus view $\rR s_*(\psi)$ instead as a map
\[
[p^*C\otimes_{\sO_V}\Kos(p^*V^\vee, s_{\can}^\vee)] \lotimes  [p^*C\otimes_{\sO_V}\Kos(p^*V^\vee, s_{\can}^\vee)] \to \omega_{V/B} \otimes L[n+r];
\]
tracing through its definition, one finds that this map is given by the composition
\begin{multline*}
[p^*C\otimes_{\sO_V}\Kos(p^*V^\vee, s_{\can}^\vee)] \lotimes  [p^*C\otimes_{\sO_V}\Kos(p^*V^\vee, s_{\can}^\vee)] \\\xr{\sim}
[p^*C\lotimes p^*C]\otimes[\Kos(p^*V^\vee, s_{\can}^\vee)\otimes \Kos(p^*V^\vee, s_{\can}^\vee)]
\\\xr{p^*\psi\otimes \phi_V} \omega_{X/B} \otimes L[n] \otimes p^*\det^{-1}V[r] \xr{\sim} \omega_{V/B} \otimes L[n+r].
\end{multline*} 
As this is exactly $p^*(C, \phi) \otimes \th_{V,\rho}$, we see that the Calm\`es-Hornbostel pushforward for $s$ is the same as that defined by the  Panin-Walter $\SL$-orientation on $\BO_\eta$, which verifies the hypothesis in our Theorem~\ref{thm:SLUniqueness}. 

Having verified this, we can prove our main result.

\begin{theorem}
  \label{thm:EulerCharKO}
Let $B$ be a regular noetherian separated scheme of finite Krull dimension with $2$ invertible in $\Gamma(B, \sO_B)$.  Let $X$ be a smooth projective $B$-scheme. Then:
  \begin{enumerate}[label=(\arabic*)]
  \item We have
    \[
      u^{\BO_\eta}(\chi(X/B))=(\oplus_{i,j=0}^{\dim_kX}\rR^i\pi_{X*}\Omega_{X/B}^j)[j-i], \Tr)
    \]
    in $\BO_\eta^{0,0}(B) \simeq \W(\D_\perf(B)) \simeq \W(B)$.
  \item Let $f:\GW(B)\to \K_0(B)$ denote the forgetful map discarding the symmetric bilinear form. Suppose that the map
    \[
      (f, q):\GW(B)\to \K_0(B)\times W(B)
    \]
    is injective (this is the case if for example  $B$ is the spectrum of a local ring). Then 
    \[
      u^{\BO}(\chi(X/B))=(\oplus_{i,j=0}^{\dim_BX}\H^i(X, \Omega_{X/B}^j)[j-i], \Tr)
    \]
    in $\GW(\D_\perf(B)) \simeq \GW(B)$.
  \item Suppose $B$ is in $\Sm_k$ for $k$ a perfect field. Then the image $\widetilde{\chi(X/B)}$ of $\chi(X/B)$ in $\pi_{0,0}(1_B)(B) \simeq \H^0(B, \sGW)$ is given by  
    \[
      \widetilde{\chi(X/B)}=(\oplus_{i,j=0}^{\dim_BX}\rR^i\pi_{X*}\Omega_{X/B}^j)[j-i], \Tr) \in
      \H^0(B, \sGW).
    \]
    In particular, if $B=\Spec k$, then 
    \[
      \chi(X/k)=(\oplus_{i,j=0}^{\dim_BX}\H^i(X, \Omega_{X/k}^j)[j-i], \Tr)\in \GW(k).
    \]
  \end{enumerate}
\end{theorem}

\begin{proof}
  The first statement follows from our comparison theorem~\ref{thm:SLUniqueness}, as detailed above, together with the motivic Gau{\ss}-Bonnet theorem~\ref{thm:GaussBonnet}. Statement (2) follows from (1) and our result for algebraic K-theory, Theorem~\ref{thm:EulerKThy}. Finally, (3) follows from (2), after we check that the unit map $u^\BO$ induces    the identity map on $\GW(k)$ via
  \[
    \GW(k)\xymatrix{\ar[r]^{(\hbox{\tiny Morel})}_\sim&} 1_k^{0,0}(\Spec k)\xymatrix{\ar[r]^{u^\BO}_\sim&} \BO^{0,0}(\Spec k)\simeq \KO^{[0]}_0(k)=\GW(k),
  \]
  where the first isomorphism arises from  Morel's theorem \cite[Theorem 6.4.1, Remark 6.4.2]{MorelICTP}, \cite[Theorem 6.40]{MorelA1} identifying $1_k^{0,0}(\Spec k)$ with $\GW(k)$. The one dimensional forms $\<\lambda\>\in \GW(k)$, $\lambda\in k^\times$,  generate $\GW(k)$,  and via Morel's isomorphism $\<\lambda\>$ maps to the automorphism of $1_k$ induced  by the automorphism $\phi_\lambda:\P^1_k\to \P^1_k$, $\phi_\lambda((x_0:x_1))= (x_0:\lambda\cdot x_1)$.   By \cite[Corollary 6.2]{AnanWitt}, the image of $\phi_\lambda$ under the unit map $u^\BO$ is also $\<\lambda\>$, after the canonical identification $\BO^{0,0}(\Spec k)\simeq \KO^{[0]}_0(k) \simeq \GW(k)$.
\end{proof}

\begin{corollary}
  \label{cor:EulChar1}
  Let $k$ be a perfect field of characteristic different from two. Let $H \in \GW(k)$ denote the class of the hyperbolic form $x^2-y^2$. Let $X$ be a smooth and projective $k$-scheme.
  \begin{enumerate}[label=(\arabic*)]
  \item Suppose $X$ has odd dimension $2n-1$. Let
    \[
      m := \sum_{i+j<2n-1}(-1)^{i+j}\dim_k\H^i(X,\Omega_{X/k}^j)- \sum_{\hbox to 30pt{\vbox{\noindent \tiny$0\le i< j$\\$i+j= 2n-1$}}} \dim_k\H^i(X,\Omega_{X/k}^j).
    \]
    Then $\chi(X/k)=m\cdot H \in \GW(k)$.
  \item Assume $X$ has even dimension $2n$. Let
    \[
      m := \sum_{i+j<2n}(-1)^{i+j}\dim_k\H^i(X,\Omega_{X/k}^j)+ \sum_{\hbox to 30pt{\vbox{\noindent \tiny$0\le i< j$\\$i+j= 2n$}}} \dim_k\H^i(X,\Omega_{X/k}^j)
    \]
    and let $Q$ be the symmetric bilinear form
    \[
      \H^n(X, \Omega^n_{X/k})\times \H^n(X, \Omega^n_{X/k}) \xr{\cup}\H^{2n}(X, \Omega^{2n}_{X/k}) 
      \xr{\Tr} k.
    \]
    Then $\chi(X/k)=m\cdot H+Q \in \GW(k)$.
  \end{enumerate}
\end{corollary}

\begin{proof}
  For $V$ a finite dimensional $k$-vector space and $n \in \Z$, we have the symmetric bilinear form in $\D_\perf(k)$
  \[
    h_n: (V[n]\oplus V^\vee[-n]) \otimes (V[n]\oplus V^\vee[-n]) \to k
  \]
  whose restriction to $V[n]\otimes V^\vee[-n]$ is the canonical pairing of $V[n]$ with $V^\vee[-n]\simeq V[n]^\vee$, and $(-1)^n$ times this pairing on $V^\vee[-n]\otimes V[n]$. The corresponding class of $h_n$ 
  in $\GW(k)$ is $(-1)^n$ times the class of $h_0$, as $(V[n]\oplus V^\vee[-n], h_n)$ is the image of the class of $V[n]$ in $\K_0(k)$ under the hyperbolic map $H:\K_0(-)\to \KO^{[0]}_0(-)$ (see e.g. \cite[Theorem 2.6]{Walter}), and $[V[n]]=(-1)^n[V[0]]$ in $\K_0(\D_\perf(k))\simeq \K_0(k)$. 

  With this in mind, we may deduce the claim from the formula
  \[
    \chi(X/k)=(\oplus_{i,j=0}^{\dim_BX}\H^i(X, \Omega_{X/k}^j)[j-i], \Tr)
  \]
  of Theorem~\ref{thm:EulerCharKO}. Indeed, in the case $\dim X=2n-1$, the symmetric bilinear form $\Tr$ is the sum of the ``hyperbolic'' forms as above on
  \[
    \H^i(X,\Omega_{X/k}^j)[j-i]\oplus \H^{2n-1-i}(X,\Omega_{X/k}^{2n-1-j})[i-j]
  \]
  for $i+j<2n-1$, or $0\le i<j$ and $ i+j= 2n-1$; and the argument in the even dimensional case is the same, except that one has the remaining factor coming from the symmetric pairing on $\H^n(X, \Omega^n_{X/k})$.
\end{proof}

The next result was obtained in 1974 independently by Abelson \cite[Theorem 1]{Abelson} and Kharlamov \cite[Theorem A]{Kharlamov} using an argument of Milnor's relying on the Lefschetz fixed point theorem.

\begin{corollary}\label{cor:EulChar2}
  Let $k$ be a field equipped with an embedding $\sigma : k \inj \R$. Let $X$ be a smooth projective $k$-scheme of even dimension $2n$. Then
  \[
    |\chi^\top(X(\R))|\le \dim_k \H^n(X, \Omega^n_{X/k}).
  \]
\end{corollary}

\begin{proof}
  We know that $\chi^\top(X(\R))$ is the signature of $\sigma_*(\chi(X/k)) \in \GW(\R)$ (see \cite[Remarks 1.11]{LevEnum}). The description of $\chi(X/k)$ given by Corollary~\ref{cor:EulChar1} gives the desired inequality
  \[
    |\operatorname{sig}\sigma_*(\chi(X/k))|\le \dim_k \H^n(X, \Omega_{X/k}). \qedhere
  \]
\end{proof}

\begin{remark}
  Let $k$ be a perfect field of characteristic different from two. The formula for the Euler characteristic given in Theorem~\ref{thm:EulerCharKO} shows that the invariant $\chi(X/k)$ is ``motivic'' in the following sense. Let $X$ and $Y$ be smooth projective $k$-schemes of respective even dimensions $2n$ and $2m$ and let $\alpha:X\dasharrow Y$ be a correspondence with $k$-coefficients of degree $n$, that is, an element $\alpha\in \CH^{m+n}(X\times Y)_k$. The correspondence $\alpha$ induces the map of $k$-vector spaces $\alpha^*:\H^m(Y,\Omega^m_{Y/k})\to \H^n(X, \Omega^n_{X/k})$. Suppose that $\alpha^*$ is an isomorphism and is compatible with the trace pairings on $\H^m(Y,\Omega^m_{Y/k})$ and $\H^n(X, \Omega^n_{X/k})$ appearing in Corollary~\ref{cor:EulChar1}. Then $\chi(X/k)=\chi(Y/k)$ in the Witt ring $\W(k)$. 

  For instance, supposing $k$ has characteristic zero, if the motives of $X$ and $Y$ (for homological equivalence with respect to de Rham cohomology) have a K\"unneth decomposition
  \[
    h(X) \simeq \oplus_{i=0}^{2n}h^i(X)\<i\>,\quad h(Y) \simeq \oplus_{i=0}^{2m}h^i(Y)\<i\>
  \]
  and $\alpha$ induces an isomorphism $\alpha^*:h^m(Y)\<m\>\to h^n(X)\<n\>$, compatible with the respective intersection products
  \[
    h^m(Y)\<m\>\otimes h^m(Y)\<m\>\to h^{2m}(Y)\<2m\>\xr{\pi_{Y*}}h^0(k) \simeq \Q,
  \]
  \[
    h^n(X)\<n\>\otimes h^n(X)\<n\>\to h^{2n}(X)\<2n\>\xr{\pi_{X*}}h^0(k) \simeq \Q,
  \]
  then $\chi(Y/k)=\chi(X/k)$ in $\W(k)$. 

  Presumably, merely having an isomorphism of motives $h^n(X)\<n\>\simeq h^m(Y)\<m\>$ would not suffice to yield $\chi(Y/k)=\chi(X/k)$ in $\W(k)$, but we do not have an example.
\end{remark}

\subsection{Descent for the motivic Euler characteristic}\label{subsec:Descent}

Let $k$ be a perfect field of characteristic different from two. With the explicit formula for $\chi(X/k)$ given by Theorem~\ref{thm:EulerCharKO}, we may find $\chi(X/k)$ for forms $X$ of some $k$-scheme $X_0$ by the usual twisting construction; this works for all manners of descent but we confine ourselves to the case of Galois descent here.

Let $X_0, X$ be smooth projective $k$-schemes of even dimension $2n$. Let $K$ be a finite Galois extension field of $k$ with Galois group $G$. Let $X_K:=X\times_kK$, $X_{0K}:=X_0\times_kK$, and suppose we have an isomorphism  $\phi:X\times_kK\to X_0\times_kK$. This gives us the cocycle $\{\psi_\sigma\in \Aut_K(X_0\times_kK)\}_{\sigma\in G}$ where $\psi_\sigma:=\phi^\sigma\circ\phi^{-1}$. Letting 
\begin{align*}
&b_0:\H^n(X_0,\Omega^n_{X_0/k})\times \H^n(X_0,\Omega^n_{X_0/k})\to k,\\
&b:\H^n(X,\Omega^n_{X/k})\times \H^n(X,\Omega^n_{X/k})\to k
\end{align*}
denote the respective symmetric bilinear forms $\Tr(x\cup y)$, the isomorphism $\phi$ induces an isometry 
\[
\phi^*:(\H^n(X_{0K},\Omega^n_{X_{0K}/K}), b_{0K})\to (\H^n(X_{K},\Omega^n_{X_{K}/K}), b_{K}),
\]
and the cocycle $\{\psi_\sigma\}_{\sigma \in G}$ determines a cocycle $\{(\psi_\sigma^*)^{-1} \in \rO(b_0)(K)\}_{\sigma \in G}$. Twisting by the latter cocycle allows one to recovers $b$ from $b_0$; explicitly, this works as follows.

Firstly, as usual, one recovers the $k$-vector space $\H^n(X,\Omega^n_{X/k})$ from the $K$-vector space $\H^n(X_{0K},\Omega^n_{X_{0K}/K})$ as the $G$-invariants for the map $x\mapsto \psi_\sigma^{*-1}(x^\sigma)$. Secondly, letting $A\in \GL(\H^n(X_{0K},\Omega^n_{X_{0K}/K}))$ be a change of basis matrix comparing the $k$-forms
$\H^n(X_0,\Omega^n_{X_0/k})\subset  \H^n(X_{0K},\Omega^n_{X_{0K}/K})$ and 
$\H^n(X,\Omega^n_{X/k})\subset  \H^n(X_{0K},\Omega^n_{X_{0K}/K})$, we recover $b$ (up to $k$-isometry) as 
\[
b(x,y)=b_0(Ax, Ay)=:b_0^A(x,y).
\]

Having performed this twist at the level of symmetric bilinear forms, we may now pass Grothendieck-Witt classes to describe the Euler characteristic of $X$: namely, Corollary~\ref{cor:EulChar1}(2) gives
\[
\chi(X_0/k)=[b_0+m\cdot H],\quad
\chi(X/k)=[b_0^A+m\cdot H]
\]
in $\GW(k)$.

\begin{remark}
  In the case of a smooth projective surface $S$ with $p_g(S)=0$, over a characteristic zero field $k$, the twisting construction reduces to a computation involving $\CH^1(S_{\bar{k}})/\kern-4pt\sim_\num$ as a $\Gal(k)$-module; here $\bar{k}$ is the algebraic closure of $k$ and $\sim_\num$ is numerical equivalence. Indeed, the assumption $p_g(S)=0$ implies that the cycle  class map in Hodge cohomology
  \[
    \cyc^\Hdg:\CH^1(S_{\bar{k}})\to \H^1(S_{\bar{k}}, \Omega^1_{S/\bar{k}})
  \]
  induces an isomorphism 
  \[
    {\CH^1(S_{\bar{k}})/\kern-4pt\sim_\num} \otimes_\Z\bar{k}\xr{\sim} \H^1(S_{\bar{k}}, \Omega^1_{S/\bar{k}}) \simeq \H^1(S, \Omega^1_{S/k})\otimes_k\bar{k}
  \]
  and the cycle class map $\cyc^\Hdg$
  transforms the intersection product on $\CH^1(S_{\bar{k}})/\kern-4pt\sim_\num$ to the quadratic form $b_0$ on $\H^1(S_{\bar{k}}, \Omega^1_{S/\bar{k}})$, induced by cup product and the trace map.  Thus, our quadratic form $b$ on $\H^1(S, \Omega^1_{S/k})$ is equivalent to the one gotten by twisting the $\bar{k}$-linear extension of the intersection product on $\CH^1(S_{\bar{k}})/\kern-4pt\sim_\num$ by the natural Galois action.

  Analogous comments hold for a ``geometrically singular'' variety, by which we mean a smooth projective $k$-scheme $X$ of dimension $2n$ such that $\H^n(X_{\bar{k}},\Omega^n_{X_{\bar{k}}/\bar{k}})$ is spanned by cycle classes,  where we replace $\CH^1/\kern-4pt
  \sim_\num$ with $\CH^n/\kern-4pt\sim_\num$.  For example, one could take a K3 surface with Picard rank 20 over $\bar{k}$ or a cubic fourfold $X$ with $\H^2(X_{\bar{k}},\Omega^2_{X_{\bar{k}}/k})\simeq \bar{k}^{21}$ spanned by algebraic cycles.
\end{remark}

\begin{examples}
  \begin{enumerate}[label=(\arabic*),leftmargin=*]
  \item As as simple example, take $S$ to be a quadric surface in $\P^3_k$ defined by a degree two homogeneous form $q(X_0,\ldots, X_3)$; we may assume that $q$ is a diagonal form, 
    \begin{multline*}
      q(X_0,\ldots, X_3)=a_0X_0^2+\sum_{i=1}^3a_iX_i^2\\=a_0(X_0+\sqrt{-a_1/a_0}X_1)(X_0-\sqrt{-a_1}X_1)\\
      +a_2(X_2-\sqrt{-a_3/a_2}X_3)(X_2+\sqrt{-a_3/a_2}X_3).
    \end{multline*}
    This trivializes $\CH^1(S)$ over $K:=k(\sqrt{-a_0a_1}, \sqrt{-a_2a_3})$, namely $\CH^1(S)=\Z\ell_1\oplus\Z\ell_2$ with $\ell_1$ defined by $(X_0-\sqrt{-a_1}X_1)=(X_2-\sqrt{-a_3/a_2}X_3)=0$ and $\ell_2$ defined by $(X_0-\sqrt{-a_1}X_1)=(X_2+\sqrt{-a_3/a_2}X_3)=0$. Embedding $\Gal(K/k)\subset \Gal(k(\sqrt{-a_0a_1})/k)\times  \Gal(k(\sqrt{-a_2a_3})/k)=\<\sigma_1\>\times\<\sigma_2\>$, the  Galois action is given by $\sigma_1(\ell_1, \ell_2)=\sigma_2(\ell_1, \ell_2)= (\ell_2, \ell_1)$. A Galois-invariant basis is thus given by $((\ell_1+\ell_2), \sqrt{a_0a_1a_2a_3}(\ell_1-\ell_2))$, and the intersection form in this basis has matrix
    \[
      \begin{pmatrix}
        2&0\\
        0&-2a_0a_1a_2a_3
      \end{pmatrix}.
    \]
    In other words, $\chi(S/k)=\<2\>+\<-2a_0a_1a_2a_3\>$.

  \item Suppose $S$ is the blowup of $\P^2_k$ along a 0-dimensional closed subscheme $Z\subset \P^2_k$, with $Z$ \'etale over $k$. Let $\ell$ denote the class of a line in $\CH^1(\P^2)$. Writing $Z_{\bar{k}}=\{p_1,\ldots, p_r\}$, we have $\CH^1(S_{\bar{k}}) \simeq \Z\cdot \ell \oplus (\oplus_{i=1}^r\Z\cdot p_r)$, with the evident Galois action and with intersection form the diagonal matrix $(1, -1,\ldots, -1)$. It is then easy to show that the twisted quadratic form $\chi(S/k)$ is $\<1\>-\Tr_{Z/k}(\<1\>)$.
  \end{enumerate}
\end{examples}

These last two examples have been computed  by different methods before: (1) is a special case of \cite[Corollary 13.2]{LevEnum} and (2) is a special case of \cite[Proposition 1.10]{LevEnum}. Here is a more interesting example.

\begin{example}
  Let $\pi:S\to C$ be a a conic bundle over a smooth projective curve $C$, all defined over $k$; we assume for simplicity that $k\subset \C$. Let $Z\subset C$ be the degeneracy locus of $\pi$: that is, $Z$ is the reduced proper closed subscheme of $C$ over which $\pi$ is not smooth. For each geometric point $z$ of $Z$, the fiber $\pi^{-1}(z)$ is isomorphic to two distinct lines in $\P^2$: $\pi^{-1}(z)=\ell_z\cup \ell_z'$. There is a ``double section'' $D\subset S$ with $D\to C$  a finite degree two morphism, and with $D\cdot \ell_z=1=D\cdot \ell_z'$ for all $z\in Z(\bar{k})$. 

  Over $\bar{k}$, the bundle $S$ is isomorphic to the blow-up of a $\P^1$-bundle $\bar{S}_{\bar{k}}\to C_{\bar{k}}$ along a finite set $Z'\subset \bar{S}_{\bar{k}}$ with $Z'\xr{\sim} Z_{\bar{k}}$ via $\pi$.

  Suppose $Z_{\bar{k}}=\{z_1,\ldots z_r\}$.  If we fix a closed point $c_0\in C\setminus Z$ of degree $d$ over $k$, we have the following basis for $\CH^1(S_{\bar{k}})_\Q/\kern-4pt \sim_\num$: 
  \[
    \ell_{z_1}-\ell_{z_1}',\ldots, \ell_{z_r}-\ell_{z_r}', D,  \pi^{-1}(c_0).
  \]
  We have the finite degree two extension $p:\tilde{Z}\to Z$, where for each $z\in Z$, $p^{-1}(z)$ corresponds to the pair of lines $\ell_z, \ell_z'$. Let $L:=k(\{z_1,\ldots, z_r\})\supset k$ and let $G:=\Aut(L/k)$. Writing $k(\tilde{Z})=k(Z)(\sqrt{\delta})$ for some $\delta\in \sO_Z^\times$, we have a basis of $\CH^1(S_L)_\Q/\kern-4pt \sim_\num$ given by
  \[
    v_1,\ldots, v_r,  D,  \pi^{-1}(c_0),
  \]
  with $v_i:=\sqrt{d}(\ell_{z_i}-\ell_{z_i}')$. The intersection form on $\<v_1,\ldots, v_r\>$ is the diagonal matrix $(-4\delta(z_1),\ldots, -4\delta(z_r))$, the subspaces $\<v_1,\ldots, v_r\>$ and $\<D,  \pi^{-1}(c_0)\>$ are perpendicular and $\<D,  \pi^{-1}(c_0)\>$ is hyperbolic. Moreover, the automorphism group $\Aut(L/k)$ acts on $\<v_1,\ldots, v_r\>$ just as it does on $\<z_1,\ldots, z_r\>$. From this it follows that the twisted intersection form $b$ is given by
  \[
    b=H-\Tr_{k(Z)/k}(\<\delta\>),
  \]
  and hence
  \[
    \chi(S/k)=m\cdot H-\Tr_{k(Z)/k}(\<\delta\>)
  \]
  with 
  \[
    m=2-\dim_k\H^0(S,\Omega^1_{S/k})-\dim_k\H^1(S,\sO_X)=\dim_\Q \H^0(S^\an,\Q)-\dim_\Q \H^1(S^\an,\Q)+1,
  \]
  where $S^\an$ is the complex manifold associated to $S_\C$.

  As a particular example, we may take $S$ to be a cubic surface $V\subset \P^3_k$ with a line $\ell$. Projection from $\ell$ realizes $V$ as a conic bundle $\pi:V\to \P^1_k$, with degeneracy locus $Z\subset \P^1_k$ a reduced closed subscheme of degree 5 over $k$.   The above implies that the symmetric bilinear form $b_V$ is given by
  \[
    b_V=H-\Tr_{Z/k}(\<\delta\>)
  \]
  and computes $\chi(S/k) = 2H-\Tr_{Z/k}(\<\delta\>)$.
\end{example}

\begin{remark} In \cite{BFS}, Bayer-Fluckiger and Serre consider the finite $k$-scheme $W$ representing the 27 lines on a cubic surface $V$ and compute the trace form $\Tr_{W/k}(\<1\>)$ in \cite[Theorem 5, Interpretation 7.3]{BFS}. They identify their form $q_{6,V}$ with the trace form on $\H^1(V,\Omega_{V/k})$ and show that
\[
\Tr_{W/k}(\<1\>)=\lambda^2b_V + (\<-1\>-\<2\>)b_V+ 7 -\<-2\>.
\]

\end{remark}

 \end{document}